\documentclass[a4paper]{amsart}
\usepackage{amsmath,amsthm,amssymb, amscd, amsfonts}
\usepackage[all]{xy}
\usepackage{leftidx}
\usepackage[latin1]{inputenc}
\usepackage{enumerate}

\usepackage[dvipsnames]{xcolor}

\theoremstyle{plain}

\newtheorem{theorem}{Theorem}[section]
\newtheorem{corollary}[theorem]{Corollary}

\newtheorem{proposition}[theorem]{Proposition}
\newtheorem{lemma}[theorem]{Lemma}

\theoremstyle{definition}
\newtheorem{definition}[theorem]{Definition}

\newtheorem{example}[theorem]{Example}

\theoremstyle{remark}
\newtheorem{remark}[theorem]{Remark}

\numberwithin{equation}{section}\theoremstyle{plain}

\newcommand{\ootimes}{\overline{\otimes}}

\renewcommand{\1}{\textbf{1}}

\newcommand{\A}{{\textbf A}}
\newcommand{\Bb}{{\textbf B}}

\newcommand{\B}{{\mathcal B}}
\newcommand{\C}{{\mathcal C}}
\newcommand{\D}{{\mathcal D}}

\newcommand{\Z}{{\mathcal Z}}

\newcommand{\M}{\mathcal{M}}
\newcommand{\N}{\mathcal{N}}

\newcommand{\toto}{\longrightarrow}

\newcommand{\E}{{\mathcal E}}
\newcommand{\U}{{\mathcal U}}

\newcommand{\Aa}{\mathbb A}

\newcommand{\Ss}{\mathcal S}

\newcommand{\Rep}{\operatorname{Rep}}

\newcommand{\K}{\mathcal{K}}
\newcommand{\KER}{\mathfrak{Ker}}
\newcommand{\modd}{\!\!-\!\operatorname{mod}}
\newcommand{\comod}{\operatorname{comod}\!-\!}

\newcommand\coend{\operatorname{coend}}
\newcommand\Aut{\operatorname{Aut}}
\newcommand\Irr{\operatorname{Irr}}
\newcommand\Fun{\operatorname{Fun}}

\newcommand\vect{\operatorname{Vec}}

\newcommand\id{\operatorname{id}}

\newcommand\co{\operatorname{co}}

\newcommand\op{\operatorname{op}}
\newcommand\Hom{\operatorname{Hom}}

\newcommand\dys{\operatorname{dys}}
\newcommand{\uno}{ \mathbf{1}}

\begin{document}
\title[Extensions of tensor categories by finite group fusion categories]{Extensions of tensor categories by finite group fusion categories}
\author{Sonia Natale}
\address{Facultad de Matem\'atica, Astronom\'\i a,  F\'\i sica y Computaci\' on.
Universidad Nacional de C\'ordoba. CIEM -- CONICET. Ciudad
Universitaria. (5000) C\'ordoba, Argentina}
\email{natale@famaf.unc.edu.ar
\newline \indent \emph{URL:}\/ http://www.famaf.unc.edu.ar/$\sim$natale}

\thanks{Partially supported by  CONICET and SeCYT--UNC}



\date{\today}

\begin{abstract} We study exact sequences of finite tensor categories of the form $\Rep G \to \C \to \D$,  where $G$ is a finite group. We show that, under suitable assumptions, there exists a group $\Gamma$ and mutual actions by permutations $\rhd: \Gamma \times G \to G$ and  $\lhd: \Gamma \times G \to \Gamma$ that make $(G, \Gamma)$ into matched pair of groups endowed with a natural crossed action on $\D$ such that $\C$ is equivalent to a certain associated crossed extension $\D^{(G, \Gamma)}$ of $\D$. Dually, we show that an exact sequence of finite tensor categories $\vect_G \to \C \to \D$ induces an $\Aut(G)$-grading on $\C$ whose neutral homogeneous component is a $(Z(G), \Gamma)$-crossed extension of a tensor  subcategory of $\D$.  As an application we prove  that such extensions $\C$ of $\D$ are weakly group-theoretical fusion categories if and only if $\D$ is a weakly group-theoretical fusion category.
In particular, we conclude that every semisolvable semisimple Hopf algebra is weakly group-theoretical.
\end{abstract}

\maketitle

\section{Introduction}

Let $\C$, $\C'$, $\C''$ be tensor categories over a field $k$. Recall from  \cite{tensor-exact} that an exact sequence of tensor categories is a sequence of
tensor functors
\begin{equation}\label{ex-sq}\xymatrix{\C' \ar[r]^f & \C \ar[r]^F &
	\C''}
\end{equation}
where $F$ is dominant and normal and $f$ is a full embedding that induces an equivalence between $\C'$ and the tensor subcategory $\KER_F$ of $\C$ of all objects $X$ such that $F(X)$ is a trivial object of $\C''$. The precise meaning of these notions are recalled in Section \ref{section-exact}.

Every exact sequence \eqref{ex-sq} gives rise to a tensor functor from $\C'$ to the category $\vect$ of finite dimensional vector spaces and therefore, by Tannakian reconstruction arguments, it induces a Hopf algebra $H$ and a canonical equivalence of tensor categories $\C' \cong \comod H$, where $\comod H$ is the category of finite dimensional right comodules over $H$.

\medbreak We shall work over an algebraically
closed field $k$ of characteristic zero.  The main result of this paper is a classification result for exact sequences of finite tensor categories \eqref{ex-sq} whose induced Hopf algebra $H$ is finite dimensional and commutative: we call such an exact sequence an \emph{abelian exact sequence}. Equivalently, 
an abelian exact sequence is an exact sequence \eqref{ex-sq} such that $\C' = \Rep G$ is the tensor category of finite dimensional representations of a finite group $G$, and the induced tensor functor $\Rep G \to \vect$  is monoidally isomorphic to the forgetful functor. Some characterizations of abelian exact sequences are given in Section \ref{section-abelian}.

\medbreak
Examples of abelian exact sequences of tensor categories arise from equivariantization under the action of a finite group on a tensor category and also from (strict) Hopf algebra extensions of dual group algebras, that is, Hopf algebra extensions $k \toto k^G \toto H \toto H'' \toto k$, where $G$ is a finite group.

Further examples of abelian exact sequences were constructed in \cite{crossed-action} from crossed actions of matched pairs of finite groups on a tensor category, as follows. Recall that a \emph{matched pair of groups} is a pair $(G, \Gamma)$, where $G$ and $\Gamma$
are groups  endowed with  mutual  actions by permutations
$\Gamma \overset{\vartriangleleft}\longleftarrow \Gamma \times G
\overset{\vartriangleright}\longrightarrow G$ such that, for all $s, t \in \Gamma$,
$g, h \in G$,
\begin{equation}\label{matched}
s \vartriangleright gh  = (s \vartriangleright g) ((s
\vartriangleleft g) \vartriangleright h), \quad st
\vartriangleleft g  = (s \vartriangleleft (t \vartriangleright g))
(t \vartriangleleft g). \end{equation}

Given a matched pair  of finite groups $(G, \Gamma)$, a \emph{$(G, \Gamma)$-crossed action} on a tensor category $\D$ consists of a right action by $k$-linear autoequivalences $\rho: \underline{G}^{op} \to \Aut_k(\D)$ and a $\Gamma$-grading $\C = \bigoplus_{s  \in \Gamma} \D_s$ of $\D$ such that for all $g \in G$, $s \in \Gamma$, $\rho^g(\D_t) \subseteq \D_{s \lhd g}$, and there exist natural isomorphisms $\gamma^g_{X, Y}: \rho^g(X \otimes Y) \to \rho^{s\rhd g}(X) \otimes \rho^g(Y)$, $X \in \D$, $Y \in \D_s$, satisfying appropriate conditions. See Section \ref{section-crossed}.

A $(G, \Gamma)$-crossed action on a tensor category $\D$ gives rise to a tensor category $\D^{(G, \Gamma)}$ that fits into a canonical abelian exact sequence of tensor categories
\begin{equation}\label{sec-crossed} \Rep G \toto \D^{(G, \Gamma)} \toto \D.
\end{equation}
We call  $\D^{(G, \Gamma)}$ a \emph{crossed extension} of the neutral homogeneous component $\D_e$ by the matched pair $(G, \Gamma)$. 
It was shown in \cite{crossed-action} that the tensor categories of finite dimensional representations of the Kac algebras \cite{kac} associated to a matched pair $(G, \Gamma)$ arise through this construction from pointed fusion categories. 

\medbreak The first main result of the paper says that crossed extensions by matched pairs do in fact exhaust  the class of abelian exact sequences of finite tensor categories:

\begin{theorem}\label{main1} Let $G$ be a finite group and let $(\E): \; \Rep G \toto \C \toto \D$ be an abelian exact sequence of finite tensor categories. Then there exists a finite group $\Gamma$ endowed with mutual actions by permutations $\rhd: \Gamma \times G \to G$, $\lhd: \Gamma  \times G \to \Gamma$ and a $(G, \Gamma)$-crossed action on $\D$ such that $(\E)$ is equivalent to the exact sequence \eqref{sec-crossed}.
\end{theorem}

Theorem \ref{main1} is proved in Section \ref{ab-ext} (see Subsection \ref{proof-main1}). 
The proof relies on the fact that the exact sequence $(\E)$ induces a commutative algebra $(A, \sigma)$ of the Drinfeld center $\Z(\C)$ such that $\Hom_\C(\1, A) \cong k$ and $A \otimes A$ is isomorphic to a direct sum of copies of $A$ as a right $A$-module in $\C$.
Furthermore, $(A, \sigma)$ gives rise to a canonical exact sequence
\begin{equation*}\xymatrix{\langle A \rangle \ar[r] & \C \ar[r]^{F_A} &
	\C_A},
\end{equation*}
which is equivalent to  $(\E)$. Here, $\langle A \rangle$ is the tensor subcategory of $\C$ generated by $A$ and $\C_A$ is the tensor category of right $A$-modules in $\C$. See \cite[Section 6]{tensor-exact}, Subsection \ref{ss-induced}.

The assumption that $(\E)$ is abelian allows, on the one hand, to identify the group $G$ with the group $\Aut_\C(A)$ of algebra automorphisms of $A$ in $\C$. This gives rise to a natural right action by $k$-linear autoequivalences of $G$ on $\C_A$ (Subsection \ref{ss-action}).

On the other hand, the assumption allows to provide the trivial right $A$-module $A \otimes A$ with a canonical half-braiding $\tilde c: A \otimes A \otimes_A - \to - \otimes_A A \otimes A$ in such a way that $(A \otimes A, \tilde c)$ becomes a commutative algebra of the Drinfeld center $\Z(\C_A)$.  This entails in turn a faithful grading of $\C_A$ by a subgroup $\Gamma$ of the group $\Aut_A(A \otimes A)$ of algebra automorphisms of $A\otimes A$ in $\C_A$ (Subsection \ref{ss-grading}).

The relevant permutation actions $\rhd$ and $\lhd$ arise, respectively, from an identification of $G$ with the set $\Hom_A(A \otimes A, A)$ of algebra morphisms $A \otimes A \to A$ in $\C_A$, and from a canonical isomorphism, for all $g \in G$, $\rho^g(X \otimes_A Y) \cong \rho^{s\rhd g}(X) \otimes_A \rho^g(Y)$, for every object $X \in \C_A$ and for every homogeneous object $Y \in \C_A$ of degree $s \in \Gamma$. See Subsections \ref{def-rhd}, \ref{def-lhd}.

\medbreak Another main result of this paper is the following theorem. For a finite group $G$, let $\vect_G$ be the fusion category of finite dimensional $G$-graded vector spaces. Let also $Z(G)$ be the center of $G$.

\begin{theorem}\label{main2} Let $G$ be a finite group and let $\vect_G \toto \C \toto \D$ be an exact sequence of finite tensor categories. Then there exist a finite group $\Gamma$ endowed with mutual actions by permutations $\lhd: \Gamma \times Z(G) \to \Gamma$ and $\rhd: \Gamma \times Z(G) \to Z(G)$ making $(Z(G), \Gamma)$ a matched pair, and an $\Aut(G)$-grading on $\C$ whose neutral homogeneous component is a $(Z(G), \Gamma)$-crossed extension of a tensor subcategory of $\D$.
\end{theorem}

Theorem \ref{main2} is proved in Section \ref{section-pt} (see Theorem \ref{main-pt}); its proof consists in constructing an $\Aut(G)$-grading  of $\C$ whose neutral homogeneous component is a (necessarilly abelian) extension of $\Rep Z(G)$ by a tensor subcategory of $\D$, and then applying Theorem \ref{main1}.

\medbreak Recall that every matched pair of groups $(G, \Gamma)$ gives rise to a group,
denoted $G \Join \Gamma$, defined as follows: $G \bowtie \Gamma = G \times \Gamma$ with multiplication
\begin{equation}\label{def-gbgamma}
(g, s) (h, t) = (g(s \rhd h), (s\lhd h) t),\end{equation} for all $g, h \in G$, $s, t\in
\Gamma$. The group $G \bowtie \Gamma$ admits an \emph{exact factorization} into its subgroups $G \times 1 \cong G$ and $1\times \Gamma \cong \Gamma$; this condition does in fact characterize the matched pairs $(G, \Gamma)$.

\medbreak
Let $\D$ be a tensor category endowed with a crossed action of a matched pair $(G, \Gamma)$. We show that the category $\D^{(G, \Gamma)}$ is categorically Morita equivalent to a tensor category graded by the group $G \bowtie \Gamma$ with neutral homogeneous component $\D_e$ (Theorem \ref{grading-crossed}).

\medbreak
A fusion category $\C$ is called \emph{weakly group-theoretical} if $\C$ is categorically Morita equivalent to a nilpotent fusion category \cite{ENO2} (we recall these notions in Section \ref{section-pmrs}). The question raised in \cite{ENO2} about the existence of a fusion category with integer Frobenius-Perron dimension which is not weakly group-theoretical is an open problem.
It is known that every extension of fusion categories with integer Frobenius-Perron dimension has integer Frobenius-Perron dimension. However, the analogous question regarding extensions of weakly group-theoretical fusion categories is open as well.

\medbreak
Let $G$ be a finite group. We combine Theorem \ref{grading-crossed} with Theorems \ref{main1} and \ref{main2} to show that if $\C$ is a fusion category fitting into an abelian exact sequence $\Rep G \toto \C \toto \D$ or into an exact sequence $\vect_G \toto \C \toto \D$, then $\C$ is  weakly group-theoretical if and only if $\D$ is weakly group-theoretical (Corollaries \ref{wgt-ext-rep} and \ref{wgt-pt}).
As a consequence, we show in Section \ref{section-hopf} that every  semisolvable semisimple Hopf algebra, as introduced in \cite{MW}, is weakly group-theoretical (Corollary \ref{main-hopf}).

\section{Preliminaries on tensor categories and tensor functors}\label{section-pmrs}

We begin by recalling the main notions that will be relevant throughout the paper. We refer the reader to the reference \cite{EGNO} for a systematic study of tensor categories.

\medbreak
A tensor category over $k$ is a $k$-linear abelian rigid monoidal
category $\C$ such that every object of $\C$ has finite length, Hom spaces are finite dimensional, the tensor product $\otimes: \C \times \C \to \C$ is $k$-bilinear and the unit object $\1 \in \C$ is simple. A tensor category $\C$ is finite if it is finite as a $k$-linear abelian category. 
A fusion category is a finite semisimple  tensor category.
Unless explicitly stated, all tensor categories will be
assumed to be strict.

\medbreak
Let $\C$ and $\D$ be tensor categories over $k$. A \emph{tensor functor} $F: \C \to \D$ is a $k$-linear exact (strong) monoidal functor $F$.

A tensor functor $F: \C \to \D$ is \emph{dominant} if every object of $\D$ is a subobject of $F(X)$ for some object $X$ of $\C$. If $\D$ is a finite tensor category, then $F$ is dominant if and only if every object of $\D$ is a subquotient of $F(X)$ for some object $X$ of $\C$; see \cite[Lemma 2.3]{eg-emc}.

\medbreak
A \emph{left module category} over a tensor category $\C$ is a  $k$-linear abelian category $\mathcal{M}$ endowed with a $\C$-action, that is, a $k$-linear functor $\ootimes:\mathcal{C}\times\mathcal{M}\rightarrow\mathcal{M}$ exact in both variables, and natural isomorphisms
$m_{X,Y,M}:(X\otimes Y)\ootimes M\rightarrow X\ootimes(Y\ootimes M)$, $u_M:\textbf{1}\ootimes M\rightarrow M$, $X, Y \in \C$, $M \in \mathcal M$, satisfying natural associativity and unitary conditions.

If $\M, \N$ are $\C$-module categories, a $\C$-\emph{module functor} $\M \to \N$ is a $k$-linear functor $F:\M \to \N$ endowed with a natural isomorphism $F \circ \ootimes \to \ootimes \circ (\id_\C \times F)$ satisfying appropriate conditions. A module category $\M$ is called \emph{indecomposable} if it is not equivalent as a $\C$-module category to the direct sum of two nontrivial module categories.

Suppose $\C$ is a finite tensor category. A finite left $\C$-module category $\M$ is \emph{exact} if for every projective object $P\in \C$ and for every $M\in\M$, $P\otimes  M$ is a projective object of $\M$.
Let $\M$ be an  indecomposable exact $\C$-module category. Then the category $\Fun_\C(\M)$ of right exact $\C$-module endofunctors of $\mathcal{M}$ is a finite tensor category.
A tensor category $\mathcal{D}$ is \emph{categorically Morita equivalent} to $\mathcal{C}$  if  $\mathcal{D}\cong \Fun_\C(\M)^{\text{op}}$ for some exact  indecomposable $\mathcal{C}$-module category $\mathcal{M}$.

\medbreak
Recall that an \emph{algebra in $\C$} is an object $A$ endowed with morphisms $m:A \otimes A \to A$ and $\eta: \1 \to A \otimes A$ in $\C$, called the \emph{multiplication} and \emph{unit} morphisms, respectively, satisfying $m(\id_A \otimes m) = m (m \otimes \id_A)$ and $m(\id_A \otimes \eta) = \id_A = m(\eta \otimes \id_A)$.

\medbreak
Let $A$ be an algebra in $\C$. A right $A$-module in $\C$
is an object $X$ of $\C$ endowed with a morphism $\mu_X: X \otimes A \to X$ such that $\mu_X(\mu_X \otimes \id_A) = \mu_X(\id_X \otimes m)$. A morphism $f: X \to Y$ of right $A$-modules is a morphism $f$ in $\C$ such that $\mu_Y(f \otimes \id_A) = f \mu_X$. The category of right $A$-modules in $\C$ will be denoted $\C_A$.

Let $F_A: \C \to \C_\A$ be the functor such that $F_A(X) = X \otimes A$, $X \in \C$, with right $A$-module structure induced by the multiplication of $A$. Then $F_A$ is left adjoint to the forgetful functor $\C_A \to \C$.

\medbreak
Left $A$-modules and $A$-bimodules in $\C$ are defined similarly; the corresponding categories will be indicated by ${}_A\C$ and ${}_A\C_A$, respectively. In particular ${}_A\C_A$ is a monoidal category with tensor product $\otimes_A$ and unit object $A$. That is, if $X$ and $Y$ are two bimodules with left actions $\lambda_X$, $\lambda_Y$ and right actions $\mu_X$, $\mu_Y$, their tensor product
$X\otimes_A Y$ is the coequalizer of the morphisms
$$\xymatrix@C=5.5em{ X \otimes A \otimes Y \ar[r]<1ex>^{\mu_X \otimes \id_Y} \ar@<-1ex>[r]_{\quad \id_X \otimes \lambda_Y}  & X \otimes Y,}
$$ The bimodule structure of $X \otimes_A Y$ is obtained by factoring the morphisms $\id_X \otimes \mu_Y$ and $\lambda_X \otimes \id_Y$ through the canonical epimorphism $X \otimes Y \to X \otimes_A Y$.
The unit object of $\C_A$ is $A$ with bimodule structure given by  multiplication.

\medbreak
The tensor product of $\C$ endows  $\C_A$ with a left $\C$-module category structure where the action $\ootimes : \C \times \C_A \to \C_A$ is defined as $X \ootimes Y = X \otimes Y$, with $$\mu_{X \ootimes Y} = \id_X \otimes \mu_Y: X\otimes Y \otimes A \to X \otimes Y,$$ for all objects $X \in \C$, and $\mu_Y: Y \otimes A \to Y$ in $\C_A$. In addition, there is an equivalence $\Fun_\C(\C_A)^{op} \cong {}_A\C_A$.

\subsection{Braidings and commutativity}\label{braid}
A \emph{braided tensor category} is a tensor category $\B$ equipped with a natural isomorphism $c_{X, Y}: X \otimes Y \to Y \otimes X$, $X, Y \in \B$, called the  \emph{braiding} of $\C$, such that
\begin{equation*}c_{X, Y \otimes Z} = (\id_Y \otimes c_{X, Z}) (c_{X, Y} \otimes \id_Z), \quad c_{X \otimes Y, Z} = (c_{X, Z}\otimes \id_Y) (\id_X \otimes c_{Y, Z}),
\end{equation*}
for all $X, Y, Z \in \C$.

An object $X$ of a braided tensor category $\B$ will be called \emph{symmetric} if $c_{X, X}^2 = \id_{X \otimes X}$.

\medbreak
Let $\C$ be a tensor category. A \emph{half-braiding} on an object $Z$ of $\C$ is a natural isomorphism $c: Z \otimes - \to - \otimes Z$ such that
\begin{equation*}c_{X \otimes Y} = (\id_X \otimes c_Y) \; (c_X \otimes \id_Y),
\end{equation*} for all objects $X, Y \in \C$. The Drinfeld center $\Z(\C)$ of $\C$ is the categoy whose objects are pairs $(Z, c)$, where $Z \in \C$ and $c$ is a half-braiding on $Z$, and morphisms $\zeta: (Z, c) \to (Z', c')$ are morphisms $\zeta: Z \to Z'$ in $\C$ such that, for all $X\in \C$,
\begin{equation*}(\id_X \otimes f) c_X = c'_X (f \otimes \id_X),
\end{equation*}

The Drinfeld center of $\C$ is a braided tensor category with monoidal tensor product and braiding defined in the form
\begin{equation*}(Z, c) \otimes (Z', c')  = (Z \otimes Z', (c \otimes \id_Z)(\id_Z \otimes c')),\quad
c_{(Z, c),  (Z', c')} = c_{Z'},
\end{equation*}
for all $(Z, c), (Z', c') \in \Z(\C)$.
The forgetful functor $\Z(\C) \to \C$, $(Z, c) \mapsto Z$, is a strict tensor functor.

\medbreak
An algebra $A$  in a braided tensor category $\B$ is called \emph{commutative} if $m \, c_{A, A} = m: A \otimes A \to A$.

\medbreak
An algebra $(A, \sigma)$ in $\Z(\C)$ will be called a \emph{central algebra} of $\C$. If $(A, \sigma)$ is a commutative central algebra of $\C$, then $\C_A$ is a monoidal category.  The half-braiding $\sigma$ gives rise to a section $\C_A \to {}_A\C_A$ of the forgetful functor; namely, for every right $A$-module $Y$,  $\lambda_Y = \mu_Y \sigma_Y: A \otimes Y \to Y$ is a left $A$-action that makes $Y$ into an $A$-bimodule.  Thus $\C_A$ becomes a monoidal subcategory of  ${}_A\C_A$. The tensor product
$X\otimes_A Y$ of two right $A$-modules $\mu_X: X \otimes A \to X$ and $\mu_Y: Y \otimes A \to Y$ is the coequalizer of the morphisms
$$\xymatrix@C=5.5em{ X \otimes A \otimes Y \ar[r]<1ex>^{\mu_X \otimes \id_Y} \ar@<-1ex>[r]_{\quad \id_X \otimes \mu_Y\sigma_Y}  & X \otimes Y.}
$$

The functor $F_A: \C \to \C_A$, $F_A(X) = X \otimes A$ is a monoidal functor with monoidal structure $(F_A^2)_{X, Y}: F_A(X)\otimes_A F_A(Y) \to F_A(X \otimes Y)$ induced by the composition
$$(\id_{X \otimes Y} \otimes m) \, (\id_X \otimes \sigma_Y \otimes \id_A): X \otimes A \otimes Y \otimes A \to X \otimes Y \otimes A,$$
for all $X, Y \in \C$, and $F_A^0 = \id_A: A \to F_A(\1)$.

\medbreak
Let $A$ be a commutative algebra in a braided tensor category $\B$. Then $(A, c_{A, -})$ is a commutative algebra in the center $\Z(\B)$ and therefore $\B_A$ is a monoidal category with tensor product $\otimes_A$ and unit object $A$. A \emph{dyslectic right $A$-module in $\B$} is a right
$A$-module $\mu_X: X \otimes A \to X$ in $\B$ satisfying $\mu_X c_{A, X}c_{X, A} = \mu_X$. The category $\dys \B_A$ of dyslectic $A$-modules in $\B$ is a full monoidal  subcategory of $\B_A$ and it is braided with braiding $\tilde c$ such that for all $X, Y \in \B_A$ the following diagram commutes:
$$\xymatrix{X \otimes Y \ar[r]^{c_{X, Y}}\ar[d]_{p} & Y \otimes X \ar[d]^{p}\\
	X \otimes_A Y \ar[r]_{\tilde c_{X, Y}} & Y \otimes_A X,}
$$
See \cite{pareigis}.

\medbreak
In particular, if $\A = (A, \sigma)$ is a commutative central algebra of a tensor category $\C$, then $\dys \Z(\C)_\A$ is a braided monoidal category. By \cite[Corollary 4.5]{schauenburg}, the functor $(Z, c) \to (Z, \tilde c)$ induces an equivalence of braided monoidal categories $\dys \Z(\C)_\A \to
\Z(\C_A)$.

\subsection{Perfect tensor functors}\label{perfect} Let $\C$ and $\D$ be tensor categories. A tensor functor $F: \C \to \D$ is called \emph{perfect} if it admits an exact right adjoint \cite[Subsection 2.1]{indp-exact}.

\medbreak Suppose $F: \C \to \D$ is a dominant perfect tensor functor. Let $R: \D \to \C$ be a right adjoint of $F$, so that $R$ is faithful exact. Let also $A = R(\1)$.  Then there exists a half-braiding $\sigma$ on $A$ such that $(A, \sigma)$ is a  commutative central algebra of $\C$ which satisfies $\Hom_\C(\1, A) \cong k$. Furthermore, there is an equivalence of tensor categories $\kappa: \D \to \C_A$ such that the following diagram of tensor functors is commutative up to a monoidal natural isomorphism
$$\xymatrix{\C \ar[r]^{F} \ar[rd]_{F_A} & \D \ar[d]^{\kappa}\\ & \C_A.}$$
We shall call $(A, \sigma)$ the \emph{induced central algebra of $F$}. See \cite[Section 6]{tensor-exact}.

\begin{lemma}\label{finite-perfect} Every tensor functor $F: \C \to \D$ between finite tensor categories $\C$ and $\D$ is perfect.
\end{lemma}

\begin{proof}  The finite tensor category $\C$ is an exact right $\C$-module category with the action given by tensor product. Similarly, $\D$ is an exact $\C$-module category with the action given by $X \ootimes Y = F(X) \otimes Y$, $X \in \C$, $Y\in \D$. Since $F$ is a tensor functor, then it is a $\C$-module functor with the $\C$-module constraint defined by the monoidal structure of $F$: $(F^2_{X_1, Y})^{-1}: F(X_1 \otimes X_2) \to F(X_1) \otimes F(X_2) = F(X_1) \ootimes X_2$, $X_1, X_2 \in \C$.
	
In addition, since $\C$ and $\D$ are finite, then $F$ has a right adjoint $R$. By \cite[Section 7.12]{EGNO}, $R: \D \to \C$ is also a $\C$-module functor. Therefore, by \cite[Proposition 7.6.9]{EGNO}, $R$ is exact. This finishes the proof of the lemma.
\end{proof}

\begin{example}\label{comodh} (\cite[Example 6.3]{tensor-exact}.)
Let $H$ be a finite dimensional Hopf algebra over $k$.
The forgetful functor $U : \comod H \to \vect_k$ admits a right
adjoint $R= - \otimes H$. Let $(A,\sigma)$ be the induced central algebra of $U$. As an algebra in $\comod H$,
$A=H$ with right coaction $\Delta$. For any right $H$-comodule $V$, the half-braiding $\sigma_V: A \otimes V \to V \otimes A$ is
defined by
$$\sigma_{V}(h \otimes v) =  v_{(0)} \otimes S(v_{(1)})\, h\, v_{(2)}, \quad h \in H, v \in V.$$

In particular, if $G$ is a finite group and $U: \Rep G \to \vect$ is the forgetful functor, then the induced central algebra of $U$ is the algebra $A_0 = k^G$ with $G$-action defined by $(g.f)(h) = f(hg)$, $g, h \in G$, and the  half-braiding given by the flip isomorphism $\tau_V: A_0 \otimes V \to V \otimes A_0$, for every representation $V$. \end{example}

\subsection{Tensor functors on $\Rep G$}\label{repg}
Let $\E$ be a finite tensor category  and let $F: \E \to \vect$ be a tensor functor.

Let $(A, \sigma) \in \Z(\E)$ be the induced central algebra of $F$.  Hence   $\E_A$ is a tensor
category equivalent to $\vect$. Moreover, let $U:\E_A \to \E$ be the forgetful functor. Then the functor $\Hom_A(A, -)
\cong \Hom_\E(\1, U(-))
: \E_A \to \vect$ is  an equivalence of tensor categories (in fact, the only one up to monoidal natural isomorphism) and the following triangle of tensor functors commutes up to a monoidal natural isomorphism:
$$\xymatrix{\E \ar[r]^{F_A} \ar[rd]_{F} & \E_A \ar[d]^{\Hom_\E(\1, U(-))}\\ & \vect.}$$

Let $F_0 = \Hom_A(A, F_A(-)) \cong \Hom_\E(\1, UF_A(-)) : \E \to \vect$ and let  $\Aut_{\otimes}(F_0)$, $\Aut_{\otimes}(F)$ be the group of monoidal natural endomorphisms of $F_0$ and $F$, respectively (such endomorphisms are in fact  automorphisms, c.f. \cite[Proposition 7.1]{joyal-street}).

\medbreak
Let also $\Aut_\E(A)$ be the set of algebra endomorphisms of $A$ in $\E$.
Then $\Aut_\E(A)$ is a group with respect to composition and there is a group isomorphism $\tau: \Aut_\E(A) \to \Aut_{\otimes}(F_0)$ defined in the form
\begin{equation*}(\tau_g)_X(\zeta) = (\id_X \otimes g)\zeta,
\end{equation*} for all $g \in \Aut_\E(A)$, $X \in \E$, $\zeta \in \Hom_\E(\1, X \otimes A)$. See \cite[Remark 2.9]{mueger-crossed}. Since $F$ and $F_0$ differ by a monoidal natural isomorphism then $\Aut_\otimes(F) \cong \Aut_\otimes(F_0)$ and therefore we obtain a group isomorphism
\begin{equation}\label{iso-aut}\Aut_\E(A) \cong \Aut_\otimes(F).
\end{equation}

\begin{remark}\label{rmk-g} Let $G$ be a finite group. Let $U: \Rep G \to \vect$ be the forgetful functor and let $A_0$ be the induced central algebra of  $U$ (c.f. Example \ref{comodh}).
In this case the groups $\Aut_{\otimes}(U) \cong \Aut_G(A_0)$ are canonically isomorphic to $G$.

\medbreak Let $\E$ be a tensor category endowed with a tensor category equivalence $\tilde F: \E \to \Rep G$. Consider the tensor functor $F = U\tilde F: \E \to \vect$ and let $A \in \E$ be the induced central algebra of $F$ (up to isomorphism, the algebra $A$ corresponds to the algebra $A_0$ under the equivalence $\tilde F$). The fact that $F$ and $U$ differ by a monoidal equivalence combined with \eqref{iso-aut} entail the group isomorphisms
\begin{equation*}\Aut_\E(A) \cong \Aut_{\otimes}(F) \cong \Aut_{\otimes}(U) \cong G.
\end{equation*}
\end{remark}

\section{Exact sequences of tensor categories}\label{section-exact}

Let $\C$, $\C''$ be tensor
categories over $k$. A tensor functor $F: \C \to \C''$ is called \emph{normal}
if for every object $X$ of $\C$, there
exists a subobject $X_0 \subset X$ such that $F(X_0)$ is the
largest trivial subobject of $F(X)$.

For a tensor functor $F: \C \to \C''$, let $\KER_F$ denote the tensor
subcategory
$F^{-1}(\langle \uno \rangle) \subseteq \C$  of objects $X$ of
$\C$ such that $F(X)$ is a trivial object of $\C''$. If the functor $F$ has a right adjoint $R$, then $F$ is normal if and only if
$R(\uno)$ belongs to $\KER_F$ \cite[Proposition 3.5]{tensor-exact}.

\medbreak Let $\C', \C, \C''$ be tensor categories over $k$. An \emph{exact
	sequence of tensor categories} is a sequence of
tensor functors
\begin{equation}\label{exacta-fusion}\xymatrix{\C' \ar[r]^f & \C \ar[r]^F &
	\C''}
\end{equation}
such that the tensor functor $F$ is dominant and normal and the tensor functor
$f$ is a full embedding whose  essential image  is $\KER_F$ \cite{tensor-exact}.
We shall say that \eqref{exacta-fusion} is an \emph{exact sequence of finite tensor categories} if $\C'$, $\C$ and $\C''$ are finite tensor categories.

\medbreak
Two exact sequences of  tensor  categories
$\C_1' \overset{f_1}\toto  \C_1 \overset{F_1}\toto \C_1''$ and $\C_2' \overset{f_2}\toto  \C_2 \overset{F_2}\toto \C_2''$ are \emph{equivalent} if there exist
 equivalences of tensor categories $\varphi': \C_1' \to \C_2'$, $\varphi: \C_1 \to \C_2$ and $\varphi'': \C_1'' \to \C_2''$ such
that the the following diagram of tensor functors commutes up to monoidal natural  isomorphisms:
$$\begin{CD}\C_1' @>{f_1}>> \C_1 @>{F_1}>> \C_1'' \\
@VV{\varphi'}V @VV{\varphi}V @VV{\varphi''}V\\
\C_2' @>{f_2}>> \C_2 @>{F_2}>> \C_2''.
\end{CD}$$

\subsection{Induced Hopf algebra} The \emph{induced Hopf algebra} $H$ of an exact sequence \eqref{exacta-fusion} is defined as
\begin{equation}
H = \coend (\omega),
\end{equation}
where $\omega = \Hom_{\C''}(\1, Ff(-)): \C' \to \vect$ \cite[Subsection 3.3]{tensor-exact}. Thus there is an equivalence of tensor categories $\C' \to \comod H$ such that the following diagram of tensor functors is commutative:
$$\xymatrix{\C' \ar[r]^{\cong \quad } \ar[rd]_{\omega} & \comod H \ar[d]^{U}\\ & \vect.}$$

\subsection{Induced central algebra}\label{ss-induced} An exact sequence sequence of tensor categories \eqref{exacta-fusion} is called \emph{perfect} if $F$ is a perfect tensor functor. It follows from Lemma \ref{finite-perfect} that every exact sequence of finite tensor categories is a perfect exact sequence.

\begin{remark} Suppose that \eqref{exacta-fusion} is a perfect exact sequence of tensor categories. Since the functor $F$ has adjoints, then $\C'$ is a finite tensor category; see \cite[Proposition 3.15]{tensor-exact}. Hence in this case the induced Hopf algebra of \eqref{exacta-fusion} is finite dimensional.
\end{remark}

Let \eqref{exacta-fusion} be a perfect exact sequence and let  $(A, \sigma) \in \Z(\C)$ be the induced central algebra of $F$. Then $(A, \sigma)$ is  \emph{self-trivializing}, that is, $A \otimes A$ is a trivial object of $\C_A$. Let $\langle A \rangle$ be the smallest abelian subcategory of $\C$ containing
$A$ and stable under direct sums, subobjects and quotients. Then $F_A: \C \to \C_A$ is a normal dominant tensor functor with $\KER_{F_A} = \langle A \rangle$. Moreover, \eqref{exacta-fusion} is equivalent to the exact sequence
\begin{equation}\xymatrix{\langle A \rangle \ar[r] & \C \ar[r]^{F_A} &
	\C_A}.
\end{equation}
See \cite[Subsection 6.2]{tensor-exact}.

\begin{remark}\label{a=h} Let $(A, \sigma) \in \Z(\C)$ be the induced central algebra of a perfect exact sequence \eqref{exacta-fusion} and let $H$ be its induced Hopf algebra. Then there are equivalences of tensor categories $\langle A \rangle \cong \C'\cong \comod H$. Under these equivalences, $(A, \sigma\vert_{\langle A \rangle})$ becomes identified with the induced central algebra of the tensor functor $\omega= \Hom_{\C''}(\1, Ff(-)): \C' \to \vect$, and thus with the induced central algebra of  the forgetful functor $U: \comod H  \toto \vect$ (see Example \ref{comodh}).
\end{remark}

\subsection{Factorization of the category $\Fun_\C(\C'')$}\label{ss-fact} Consider an exact sequence of finite tensor categories \eqref{exacta-fusion}.  Let $(A, \sigma) \in \Z(\C)$ be its induced central algebra.

\medbreak
Recall from \cite[Proposition 7.3]{mn} that the category $\Fun_\C(\C'')^{op}$ admits a \emph{factorization} \begin{equation} \Fun_\C(\C'')^{op} \cong (H\modd) \boxtimes \C'',\end{equation} where $H$ is the induced Hopf algebra of \eqref{exacta-fusion}.

More precisely, let $_A\C_A$ be the tensor category of $A$-bimodules in $\C$. Then $\Fun_\C(\C_A)^{op} \cong {}_A\C_A$ and there is an equivalence of $k$-linear categories
\begin{equation}\label{f-gral} (H\modd) \boxtimes \C_A \toto {}_A\C_A.
\end{equation} Under the equivalences of tensor categories $H\modd \cong {}_A(\comod H)_A \cong {}_A{\C'}_A \subseteq {}_A\C_A$ and $\C'' \cong \C_A \subseteq {}_A\C_A$, \eqref{f-gral} is induced by the tensor product functor $\otimes_A :{}_A{\C'\!}_A \boxtimes \C_A \to {}_A\C_A$.   See \cite[Proposition 7.3]{mn}.

\section{Group actions, group gradings and crossed actions of matched pairs on tensor categories}\label{section-crossed}
Let $G$ be a finite group and let $\C$ be a tensor category. Recall that a \emph{right action of $G$ on $\C$  by $k$-linear
autoequivalences} is a monoidal functor $\rho:
\underline{G}^{\op} \to \Aut_k(\C)$.
Explicitly, for every $g \in G$, we have a
$k$-linear functor $\rho^g: \C \to \C$ and natural isomorphisms
$$\rho^{g,h}_2 : \rho^g \rho^h \to \rho^{hg}, \quad g, h \in G,$$ and
$\rho_0 : \id_\C \to \rho^e$, satisfying
\begin{align}\label{ro-2} & (\rho^{ba, c}_2)_X \, (\rho^{a,b}_2)_{\rho^c(X)} =
(\rho^{a, cb}_2)_X \, \rho^a((\rho^{b,c}_2)_X), \\ \label{ro-3}& (\rho^{a,
e}_2)_X \rho^{a}({\rho_0}_X) = \id_{\rho^a(X)} = (\rho^{e, a}_2)_X
(\rho_0)_{\rho^{a}(X)},
\end{align}
for all $X \in \C$, $a, b, c \in G$.

Let $\rho: \underline G^{\op} \to  \Aut_k(\C)$ be a right action of
$G$ on $\C$ by $k$-linear autoequivalences. The \emph{equivariantization} of $\C$ under the action $\rho$ is the $k$-linear abelian category $\C^G$ whose objects are pairs $(X, r)$, where $X$ is an object
of $\C$ and $r = (r^g)_{g \in G}$ is a collection of isomorphisms $r^g:\rho^g(X)
\to X$, $g \in G$, such that
\begin{equation}\label{deltau} r^g \rho^g(r^h) = r^{hg} (\rho^{g,
	h}_2)_X, \quad \forall g, h \in G, \quad \quad
r^e{\rho_0}_X=\id_X,\end{equation} and a morphism $f: (X, r) \to
(Y, r')$ is a morphism $f: X \to Y$ in $\C$ such
that $fr^g = {r'}^g\rho^g(f)$, for all $g \in G$.

\medbreak A \emph{$G$-grading} of $\C$ is a decomposition $\C = \bigoplus_{g \in G}\C_g$ into abelian
subcategories $\C_g$, such that $\C_g \otimes \C_h \subseteq \C_{gh}$, for all
$g, h \in G$. A $G$-grading of $\C$ is equivalent to a $G$-grading of its Grothendieck ring.

The subcategories $\C_g$, $g \in G$, are called the \emph{homogeneous components} of
the grading. The neutral homogeneous component $\C_e$ is a tensor subcategory of $\C$. A $G$-grading $\C = \bigoplus_{g \in G}\C_g$ is faithful if $\C_g \neq 0$, for all $g \in G$.

\medbreak
Every finite tensor category $\C$ has a faithful universal grading $\C = \bigoplus_{u \in U(\C)}\C_u$, with neutral homogeneous component $\C_e$ equal to the adjoint subcategory $\C_{ad}$, that is, the smallest tensor
Serre subcategory of $\C$ containing the objects $X \otimes X^*$, where $X$ runs over the simple objects of $\C$. The group $U(\C)$ is called the \emph{universal grading group} of $\C$.

The upper central series of $\C$, $\dots \subseteq \C^{(n+1)} \subseteq \C^{(n)} \subseteq \dots \subseteq \C^{(0)} = \C$, is defined as $\C^{(0)} = \C$ and $\C^{(n+1)} = (\C^{(n)})_{ad}$, for all $n \geq 0$.
A tensor category $\C$ is called \emph{nilpotent} if there exists some $n \geq 0$ such that $\C^{(n)} \cong \vect$.
See \cite[Section 3.5 and 4.14]{EGNO}, \cite{gel-nik}.

\medbreak
Let $\Irr(\C)$ denote the set of isomorphism classes of simple objects of $\C$. Let also $\partial: \Irr(\C) \to G$ be a map such that for all simple objects $X$, $Y$ and for every composition factor $Z$ of $X \otimes Y$, we have $\partial(Z) = \partial(X) \partial(Y)$. Then $\partial$ defines a $G$-grading on $\C$, $\C = \bigoplus_{g \in G} \C_g$, where for each $g \in G$, $\C_g$ is the the smallest Serre subcategory of $\C$ containing the simple objects $X$ such that $\partial(X) = g$.

\subsection{Crossed extensions of tensor categories by matched pairs of groups}\label{ss-crossed}

Let $\C$ be a tensor category over $k$ and let $(G, \Gamma)$ be a matched pair of  groups.
Recall from \cite{crossed-action} that  a $(G, \Gamma)$-\emph{crossed action}
on $\C$ consists of:

\begin{itemize}\item  A $\Gamma$-grading $\C = \bigoplus_{s \in \Gamma}
\C_s$.

\item  A right action of $G$ by $k$-linear autoequivalences $\rho:
\underline{G}^{\op} \to \Aut_k(\C)$ such that
\begin{equation}\label{rho-partial} \rho^g(\C_s) = \C_{s \lhd g},\quad \forall
g\in G, \, s\in \Gamma,\end{equation}

\item A collection of natural isomorphisms $\gamma = (\gamma^g)_{g\in G}$:
\begin{equation}\label{gamma}\gamma^g_{X, Y}: \rho^g(X \otimes Y) \to
\rho^{t\rhd g}(X) \otimes \rho^g(Y), \quad X \in \C, \, t\in \Gamma,\,  Y \in
\C_t, \end{equation}

\item A collection of isomorphisms $\gamma^g_0: \rho^g(\uno) \to \uno$, $g \in
G$.
\end{itemize}

These data are subject to the commutativity of the following diagrams:

\begin{itemize}\item[(a)] For all $g \in G$, $X \in \C$, $s, t \in \Gamma$, $Y
\in \C_s$, $Z
\in \C_t$,
\begin{equation*}
\xymatrix @C=0.6in @R=0.45in{
\rho^g(X \otimes Y \otimes Z) \ar[rr]^{\gamma^g_{X\otimes Y, Z}}
\ar[d]_{\gamma^g_{X, Y \otimes Z}} & & \rho^{t \rhd g}(X \otimes Y) \otimes
\rho^g(Z) \ar[d]^{\gamma^{t\rhd g}_{X, Y} \otimes \id_{\rho^g(Z)}} \\
\rho^{st \rhd g}(X) \otimes \rho^g(Y \otimes Z)\ar[rr]_{\id_{\rho^{st \rhd
g}(X)}
\otimes \gamma^g_{Y, Z} \qquad }&& \rho^{s \rhd (t \rhd g)}(X) \otimes \rho^{t
\rhd g}(Y) \otimes \rho^g(Z)}
\end{equation*}

\item[(b)] For all $g \in G$, $X \in \C$,
\begin{equation*}\xymatrix @C=0.6in @R=0.45in{
\rho^g(X) \otimes \rho^g(\uno)  \ar[dr]_{\id_{\rho^g(X)} \otimes \gamma^g_0
\quad } & \ar[l]_{\quad \gamma^g_{X, \uno}} \rho^g(X) \ar[d]^=
\ar[r]^{\gamma^g_{\uno, X} \quad } & \rho^g(\uno) \otimes \rho^g(X)
\ar[dl]^{\gamma^g_0 \otimes \id_{\rho^g(X)}} \\
& \rho^g(X) }\end{equation*}

\item[(c)] For all $g, h \in G$, $X \in \C$, $s\in \Gamma$, $Y \in \C_s$,
\begin{equation*}\xymatrix@C=0.6in @R=0.45in{
\rho^g\rho^h(X\otimes Y) \ar[dd]_{\rho^g(\gamma^h_{X, Y})} \ar[r]^{{\rho_2}_{X
\otimes Y}^{g, h}} & \rho^{hg}(X \otimes Y) \ar[d]^{\gamma^{hg}_{X, Y}} \\
& \rho^{s \rhd hg}(X) \otimes \rho^{hg}(Y) \\
\rho^g(\rho^{s\rhd h}(X) \otimes \rho^h(Y)) \ar[r]_{\gamma^g_{\rho^{s\rhd h}(X),
\rho^h(Y)}}  & \rho^{(s\lhd h)\rhd g}\rho^{s\rhd h}(X) \otimes \rho^g\rho^h(Y)
\ar[u]_{{\rho_2}_X^{(s\lhd h) \rhd g, s \rhd h} \otimes
{\rho_2}_Y^{g, h}} }\end{equation*}

\item[(d)] For all $g, h \in G$,
\begin{equation*}\xymatrix@C=0.6in @R=0.45in{
\rho^g\rho^h(\uno) \ar[d]_{\rho^{g}(\gamma^h_0)} \ar[r]^{(\rho_2^{g, h})_{\uno}}
& \rho^{hg}(\uno) \ar[d]^{\gamma^{hg}_0} \\
\rho^g(\uno) \ar[r]_{\gamma^{g}_0} & \uno}
\end{equation*}

\item[(e)] For all $X \in \C$, $s \in \Gamma$, $Y \in \C_s$,
\begin{equation*}\xymatrix@C=0.6in @R=0.45in{
X\otimes Y \ar[dr]_{{\rho_0}_X \otimes {\rho_0}_Y} \ar[r]^{{\rho_0}_{X\otimes
Y}} & \rho^e(X\otimes Y) \ar[d]^{\gamma^{e}_{X, Y}} \\
& \rho^e(X)\otimes \rho^e(Y)}
\xymatrix@C=0.6in @R=0.45in{
\uno \ar[dr]_= \ar[r]^{{\rho_0}_\uno} & \rho^e(\1) \ar[d]^{\gamma^e_0}
\\
& \uno}
\end{equation*}
\end{itemize}

A (finite) tensor category $\C$ endowed with a $(G, \Gamma)$-crossed action will be called a (finite) \emph{$(G, \Gamma)$-crossed tensor category}.

\medbreak Let $(G, \Gamma)$ be a matched pair of
finite groups and let $\C$ be a $(G, \Gamma)$-crossed tensor category. Then there is a canonical  exact sequence of tensor categories
\begin{equation}\label{c-abelian}\Rep G \toto \C^{(G, \Gamma)} \overset{F}\toto
\C\end{equation} with induced Hopf algebra $H \cong
k^G$. The tensor category $\C^{(G, \Gamma)}$ is defined as the equivariantization $\C^G$ of $\C$ under the action $\rho$ and $F: \C^{(G, \Gamma)} \to \C$ is the forgetful functor $F(V, (r^g)_{g \in G}) = V$. The tensor product of $\C^{(G, \Gamma)}$ is defined as $(X, r) \otimes (Y, r') = (X \otimes Y, \tilde
r)$, where  ${\tilde r}^g$, $g \in G$, is the composition
$$\bigoplus_{s \in \Gamma} \rho^g(X \otimes Y_s)
\overset{\oplus_s\gamma^g_{X, Y_s}}\toto \bigoplus_{s \in \Gamma} \rho^{s \rhd
	g}(X) \otimes \rho^g(Y_s) \overset{\oplus_s r^{s \rhd g} \otimes {r'}^g_s}\toto
\bigoplus_{s \in \Gamma} X \otimes Y_{s\lhd g} = X \otimes Y,$$
for $Y = \bigoplus_{s \in \Gamma}Y_s$, $Y_s \in \C_s$. See \cite[Theorem 6.1]{crossed-action}.

\begin{definition} The tensor category $\C^{(G, \Gamma)}$ will be called a \emph{crossed extension of $\C_e$ by the matched pair $(G, \Gamma)$} or simply a \emph{$(G, \Gamma)$-crossed extension of $\C_e$}.
\end{definition}

\begin{example} Suppose that $\rho: \underline{G}^{op} \to \Aut_{\otimes}(\C)$ is an action by tensor auto-equivalences of a tensor category $\C$. Then the equivariantization $\C^G$ is a $(G, \{e\})$-crossed extension of $\C$, where $\{e\}$ is the trivial group endowed with the trivial actions $\lhd: \{e\} \times G \to G$ and $\rhd: \{e\} \times G \to \{e\}$.
	
On the other hand, if $\C$ is a tensor category graded by a group $\Gamma$, then $\C$ is a $(\{e\}, \Gamma)$-crossed extension of $\C_e$ in a similar way.
\end{example}

Further examples of $(G, \Gamma)$-crossed extensions are the fusion categories of representations of abelian extensions of Hopf algebras  (or Kac algebras). Every such Hopf algebra $H$ is isomorphic to a suitable bicrossed product $k^\Gamma {}^\tau\#_\sigma kG$ and $H\modd$ is a $(G, \Gamma)$-crossed extension of $\vect$. See \cite[Subsection 8.2]{crossed-action}.

The following example shows that some $(G, \Gamma)$-crossed extensions cannot be built up by means of equivariantizations or group graded extensions.

\begin{example} Let $n \geq 5$ be an odd integer and let $H = k^{\Aa_{n-1}} \# kC_n$, $n \geq 5$,  be the bicrossed product associated to the matched pair $(C_n, \Aa_{n-1})$ arising from the exact factorization $\Aa_{n} = \Aa_{n-1} C_n$ of the alternating group $\Aa_n$,  where $C_n = \langle (12\dots n)\rangle$. See \cite[Section 8]{mk-ext}. So that the group $C_n \bowtie \Aa_{n-1}$ is isomorphic to $\Aa_n$ and the Drinfeld center of the category $H\modd$ is equivalent to $D(\Aa_n)\modd$, where $D(\Aa_n)$ is the Drinfeld double of the group algebra $k\Aa_n$. Then $\Z(H\modd)$ contains a unique nontrivial Tannakian subcategory $\E$ equivalent, as a braided fusion category, to $\Rep \Aa_n$ \cite[Example 3.3]{core-wgt}.
	
We claim that  if $G$ is a nontrivial finite group then $H\modd$ is not equivalent to a $G$-equivariantization or to a $G$-graded extension of any fusion category $\C$. Indeed, if this were the case, then Propositions 2.9 and 2.10 of \cite{ENO2} would force $G = \Aa_n$. But then $\C \cong \vect$, which is not possible, since $H\modd$ is not pointed and it is not equivalent to $\Rep \Aa_n = (\vect)^G$.
\end{example}

\subsection{$G \bowtie \Gamma$-grading on $\Fun_{\C^{(G, \Gamma)}}(\C)$}
Let $\C$ be a finite $(G, \Gamma)$-crossed tensor category.
As explained in Subsection \ref{ss-fact}, the exact sequence \eqref{c-abelian} gives rise to an equivalence of  $k$-linear categories
\begin{equation}\label{fact}\vect_G \boxtimes \, \C \to \Fun_{\C^{(G, \Gamma)}}(\C)^{op}.\end{equation} Let $(A, \sigma)$ be the induced central algebra of \eqref{c-abelian}. Under the equivalences of tensor categories $\vect_G \cong {}_A(\Rep G)_A$, $\C \cong (\C^{(G, \Gamma)})_A \subseteq {}_A(\C^{(G, \Gamma)})_A$, and  $\Fun_{\C^{(G, \Gamma)}}(\C)^{op} \cong {}_A(\C^{(G, \Gamma)})_A$, the equivalence \eqref{fact} is induced by the tensor product functor:
$$\xymatrix{&{}_A(\Rep G)_A \boxtimes (\C^{(G, \Gamma)})_A \ar[r]^{\qquad \otimes_A} & {}_A(\C^{(G, \Gamma)})_A.}$$

Note that the action $\rho: \underline{G}^{op} \to \Aut_k(\C)$ can be regarded as a $k$-linear action of $\vect_G$ on $\C$, that is, a $k$-linear monoidal functor $\rho: \vect_G \to \Fun_k(\C)^{op}$ (c.f. \cite[Remark 4.3]{DGNO}).

\begin{lemma}\label{emb-vecg} The embedding $\vect_G \to \Fun_{\C^{(G, \Gamma)}}(\C)^{op}$ arising from \eqref{fact} is induced by the functor $\rho: \vect_G \to \Fun_k(\C)^{op}$.
\end{lemma}

\begin{proof} First note that, for every $g \in G$, $\rho^g: \C \to \C$ is a $\C^{(G, \Gamma)}$-module functor with action constraint \begin{equation}\label{equiv-rho}
{c^g}_{(M, r), X} =  r^{t \rhd g}_M \otimes \id_{\rho^g(X)} : \rho^g(M\otimes X) \to M \otimes \rho^g(X), \end{equation}
for every simple object $X \in \C_t$, $t \in \Gamma$. Hence the (right) action $\rho$ corresponds in fact to a tensor functor $\rho: \vect_G \to \Fun_{\C^{(G, \Gamma)}}(\C)^{op}$.

Consider the tensor subcategory $\K = \KER_F \subseteq \C^{(G, \Gamma)}$, where $F:\C^{(G, \Gamma)} \to \C$ is the forgetful functor. Every object $(V, r)$ of $\K$ gives rise to a representation $\pi: G \to \operatorname{GL}(V)$ in the form $\pi(g) = r^g$, $g \in G$. This induces an equivalence of tensor categories $\K \cong \Rep G$ that allows us to identify $\K$ with $\Rep G$ in what follows. We shall also use the canonical identification of the tensor subcategory of trivial objects of $\C$ with the category $\vect$.

The restriction of the forgetful functor $F$ to $\K$ becomes thus identified with the forgetful functor $\Rep G \to \vect$. It follows from \cite[Theorem 4.2]{ostrik} that there is an equivalence of tensor categories 
\begin{equation}\label{eq-k}\vect_G \to \Fun_\K(\vect)^{op},\end{equation} which sends the one-dimensional object $k_g$ graded in degree $g$ to the $\K$-module endofunctor $a_g: \vect \to \vect$ such that $a_g(V) = V \otimes_k k_g = V$,
$$(a^2_g)_{X, M}: a_g(X \otimes M) \to X \otimes a_g(M), \quad x \otimes m \mapsto g. x \otimes m,$$
for all $X \in \K$, $M \in \vect$, $x\in X$, $m \in M$.

In addition $\rho$ restricts to the trivial action (by tensor autoequivalences) $\tilde\rho: G \to \Fun_k(\vect)^{op}$ and $\K \cong \vect^G$. By formula \eqref{equiv-rho} it gives rise to a $k$-linear monoidal functor $\tilde\rho: \vect_G \to \Fun_\K(\vect)^{op}$ such that $\tilde\rho(k_g) = {\tilde\rho}^g$ with $\K$-action constraint
\begin{equation*}
{c^g}_{(V, r), W} =  r^{g}_V \otimes \id_W : V \otimes W \to V \otimes W. \end{equation*}
The previous discussion implies that $\tilde\rho: \vect_G \to \Fun_\K(\vect)^{op}$ coincides with the tensor equivalence \eqref{eq-k}.

\medbreak Finally, using the equivalence of left $\C^{(G, \Gamma)}$-module categories $\C \cong (\C^{(G,  \Gamma)})_A$, we obtain a commutative diagram of tensor functors
$$\xymatrix{
	\vect_G \ar[r]^{\cong} \ar[d]_{=} & {}_A(\K)_A\ar[r]\ar[d]^{\cong} & {}_A(\C^{(G, \Gamma)})_A\ar[d]^{\cong}\\
	\vect_G\ar[r]_{\tilde\rho} & \Fun_\K(\vect)^{op}\ar[r]_{\varphi} & \Fun_{\C^{(G, \Gamma)}}(\C)^{op},}$$
where $\varphi(R) = \, - \otimes_A R(A)$, for all  $R \in \Fun_\K(\vect)$.  We know that $\tilde\rho: \vect_G \to \Fun_\K(\vect)^{op}$ is an equivalence. This implies the lemma since by commutativity of the diagram,
$\varphi({\tilde\rho}^g) \cong  \, - \otimes_A {\tilde\rho}^g(A) = \, - \otimes_A {\rho}^g(A) \cong \rho^g$.
\end{proof}

\begin{theorem}\label{grading-crossed} There is a $G \bowtie \Gamma$-grading on $\Fun_{\C^{(G, \Gamma)}}(\C)^{op}$ with trivial homogeneous component $\C_e$. This grading is faithful if and only if the $\Gamma$-grading of $\C$ is faithful.
\end{theorem}

\begin{proof} To prove the theorem we shall identify $\mathcal A = \Fun_{\C^{(G, \Gamma)}}(\C)^{op}$ with $\vect_G \boxtimes \C$ as $k$-linear abelian categories under the equivalence \eqref{fact}. For every $g \in G$, $t \in \Gamma$, let $\mathcal A_{(g, t)}$ be the full abelian subcategory generated by objects of the form $k_g \boxtimes X$, $X \in \C_t$. Since the tensor product of $\mathcal A$ is given by composition of endofunctors, using Lemma \ref{emb-vecg} and the isomorphisms \eqref{gamma}, we obtain
\begin{align*}
(k_g \boxtimes X) \otimes (k_h \boxtimes Y) & =  (\, - \otimes Y) \circ \rho^h \circ  (\, - \otimes X) \circ \rho^g   \\
&\cong  (\, - \otimes Y) \circ (\, - \otimes \rho^h(X)) \rho^{t\rhd h}\rho^g \\
& \cong (\, - \otimes \rho^{h}(X) \otimes Y) \circ \rho^{g(t \rhd h)} \\
& = k_{g(t \rhd h)} \boxtimes (\rho^{h}(X) \otimes Y),
\end{align*}
for all $g, h \in G$, $t, s \in \Gamma$, $X \in \C_t$, $Y\in \C_s$.
By condition \eqref{rho-partial}, $\rho^{h}(X) \otimes Y \in \C_{(t \lhd h)s}$. Thus we get that $(k_g \boxtimes X) \otimes (k_h \boxtimes Y) \in \mathcal A_{(g(t \rhd h), (t \lhd h)s)}$. This implies the theorem in view of the definition  of the product in $G \bowtie \Gamma$ (see \eqref{def-gbgamma}).
\end{proof}

Recall that a fusion category is weakly group-theoretical if it is categorically Morita equivalent to a nilpotent fusion category. The class of weakly group-theoretical fusion categories is known to be closed under taking fusion subcategories, dominant images, finite group equivariantizations and graded extensions \cite{ENO2}.

\begin{corollary}\label{wgt-crossed} The category $\C^{(G, \Gamma)}$ is a weakly group-theoretical fusion category if and only if $\C$ is a weakly group-theoretical fusion category.
\end{corollary}

\begin{proof} It follows from \cite[Proposition 6.2]{crossed-action}, $\C^{(G, \Gamma)}$ is a fusion category if and only if $\C$ is a fusion category. By Theorem \ref{grading-crossed} the category $\Fun_{\C^{(G, \Gamma)}}(\C)^{op}$ is a group graded extension of $\C_e$. Hence $\Fun_{\C^{(G, \Gamma)}}(\C)^{op}$ is weakly group-theoretical if and only if $\C_e$ is weakly group-theoretical if and only if $\C$ is weakly group-theoretical.  This implies the statement since  $\C^{(G, \Gamma)}$ is categorically Morita equivalent to $\Fun_{\C^{(G, \Gamma)}}(\C)^{op}$.
\end{proof}

\section{Abelian extensions of tensor categories}\label{section-abelian}

\begin{definition} We shall say that an exact sequence of tensor categories \eqref{exacta-fusion} is an \emph{abelian exact sequence} if its induced Hopf algebra $H$ is finite dimensional and commutative.
\end{definition}
Therefore the induced Hopf algebra of an abelian exact sequence is isomorphic to the dual group algebra $k^G$, for some finite group $G$.

\medbreak 
Examples of abelian exact sequences of tensor categories arise from exact sequences of finite dimensional Hopf algebras with cocommutative cokernel; we discuss these examples in Section \ref{section-hopf}.

\begin{example}\label{g-abelian} Let $G$ be a finite abelian group. Then every exact sequence $(\E): \C' \toto \C \toto \C''$ such that $\C' \cong \Rep G$ is an abelian exact sequence. Indeed, since $k$ is an algebraically closed field, then the simple objects of $\Rep G$ are one-dimensional and the group of invertible objects of $\Rep G$ is isomorphic to the group $\widehat{G}$ of $k$-linear characters of $G$. In addition $X \otimes Y \cong Y \otimes X$, for all objects $X, Y \in \Rep G$. If $H$ is the induced Hopf algebra of $(\E)$, then  $\comod H \cong \Rep G$ and therefore $H$ is finite dimensional cocommutative. Hence $H \cong kG(H)$, where $G(H)$ is the group of group-like elements of $H$. The equivalence $\comod H \cong \Rep G$ gives rise to a group isomorphism $G(H) \cong \widehat G$. Hence the group $G(H)$ is abelian and $H$ is commutative.
\end{example}

\begin{proposition}\label{ind-forg} An exact sequence tensor categories $(\E): \C' \to \C \to \C''$ is abelian if and only if it is equivalent to an exact sequence of the form
\begin{equation}\label{ext-rep}\Rep G \overset{f_0}\toto \C \overset{F_0}\toto \D,
\end{equation}
where $G$ is a finite group, such that there exists a monoidal natural isomorphism $\Hom_{\D}(\1, F_0f_0(-)) \cong U$, where $U: \Rep G \to \vect$ is the forgetful functor.
\end{proposition}

\begin{proof} Suppose that $(\E)$ is an exact sequence equivalent to \eqref{ext-rep}. Its induced Hopf algebra is isomorphic to that of \eqref{ext-rep} and therefore it is finite dimensional and commutative, since $\coend (U) \cong k^G$. Conversely, suppose that $(\E)$ is an abelian exact sequence and let $H$ be its induced Hopf algebra, so that $H \cong k^G$ for some finite group $G$. Hence there is a strict isomorphism of tensor categories $\comod H \cong \Rep G$. Moreover, since $\C' \cong \comod H$ is finite, then $(\E)$ is equivalent to the exact sequence $(\E'): \; \KER_{U_T} \toto (\C'')^T \overset{U_T}\toto \C''$, where $T$ is the normal Hopf monad associated to $(\E)$  \cite[Theorem 5.8]{tensor-exact}.
In addition there is an equivalence of tensor categories $f: \KER_{U_T} \to \comod H$ such that $Uf = \Hom_{\C''}(\1, U_T(-))$ (c.f. the proof of \cite[Lemma 5.4]{tensor-exact}), where $U$ is the forgetful functor on $\comod H = \Rep G$. Thus we obtain an exact sequence $(\E''): \; \Rep G \toto (\C'')^T \overset{U_T}\toto \C''$, equivalent to $(\E)$, that induces the forgetful functor on $\Rep G$. 	
This finishes the proof of the proposition.
\end{proof}

\subsection{Perfect abelian exact sequences}\label{perfect-abelian} 
Recall that if $\A = (A, \sigma)$ is a commutative central algebra of $\C$,  a dyslectic right $\A$-module in $\Z(\C)$  is a right
$\A$-module $\mu_{(Z, c)}: (Z, c) \otimes (A, \sigma) \to (Z, c)$ in $\Z(\C)$ satisfying $\mu_{(Z, c)} \; \sigma_Z c_A = \mu_{(Z, c)}$; see Subsection \ref{braid}.

\begin{proposition}\label{char-sym} Let $(\E): \C' \toto \C \toto \C''$ be a perfect exact sequence of tensor categories with  induced central algebra $\A = (A, \sigma)$. Then the following are equivalent:
	
\begin{enumerate}[(i)]
\item $\A$ is a symmetric object of $\Z(\C)$.
\item $\A \otimes \A$ is a dyslectic right $\A$-module in $\Z(\C)$.
\item $(\E)$ is an abelian exact sequence.
\end{enumerate}
\end{proposition}

\begin{proof}
Let $c_{A \otimes A, -}$ denote the half-brading in $A \otimes A$, that is, $c_{A \otimes A, X} = (\sigma_X \otimes \id_A)(\id_A \otimes \sigma_X)$, for all $X \in \C$.
A straightfoward calculation shows that
$$\mu_{A \otimes A} \; \sigma_{A\otimes A} c_{A \otimes A, A} = (\id_A \otimes m \sigma) \; (\sigma^2 \otimes \id_A) \; (\id_A \otimes \sigma).$$	
If $\sigma^2 = \id_{A \otimes A}$, the commutativity of $(A, \sigma)$ in $\Z(\C)$ implies that $$\mu_{A \otimes A} \; \sigma_{A\otimes A} c_{A \otimes A, A} = \mu_{A \otimes A}.$$ Hence (i) $\Longrightarrow$ (ii).

\medbreak
Let $H$ be the induced Hopf algebra of $(\E)$. Since $(\E)$ is perfect, the functor $\C \toto \C''$ has adjoints and therefore $H$ is finite dimensional \cite[Proposition 3.15]{tensor-exact}. 
Observe that a Hopf algebra $H$ is commutative if and only if the adjoint action $\leftharpoonup: H \otimes H \to H$, $x \leftharpoonup h = \Ss(h_{(1)}) x h_{(2)}$ is trivial. In addition, if $H$ is commutative, then $\Ss^2 = \id_H$.

Let $\sigma_0 = \sigma\vert_{\langle A \rangle}$. We can identify $(A, \sigma_0)$ with the induced central algebra of the forgetful functor $U: \comod H \to \vect$ (see Remark \ref{a=h}). Then $A = H$ with the half-braiding $\sigma_0$ described in Example \ref{comodh}, and we have
$$\sigma_0^2 (a \otimes b) = a_{(1)} \leftharpoonup b_{(2)} \otimes \Ss^2(b_{(1)}) \leftharpoonup a_{(2)}b_{(3)}, \quad a, b \in H.$$
Hence we get (iii) $\Longrightarrow$ (i). 	

\medbreak Next, for all $a, b, c \in H$, we have
$$\mu_{A \otimes A} \; (\sigma_0)_{A\otimes A} (c_0)_{A\otimes A, A} (a \otimes b \otimes c) = \Ss(c_{(2)}) a_{(1)} c_{(3)} \otimes \Ss(c_{(4)}) \Ss(a_{(2)}) \Ss^2(c_{(1)}) a_{(3)} b c_{(5)}.$$
If $\A \otimes \A$ is a dyslectic $\A$-module, the last relation implies that the adjoint action of $H$ on itself is trivial, and therefore $H$ is commutative. This shows that (ii) $\Longrightarrow$ (iii) and finishes the proof of the proposition.
\end{proof}

Recall from  \cite[Corollary 4.5]{schauenburg} that there is an equivalence of braided tensor categories $\dys \Z(\C)_\A \to
\Z(\C_A)$ that sends an object  
$(Z, c) \in \dys \Z(\C)_\A$ to the object $(Z, \tilde c) \in \Z(\C_\A)$, where the half-braiding $$\tilde c_{Z, N}: Z
\otimes_A N \to N \otimes_A Z, \quad N \in \C_\A,$$ is
induced by the half-braiding
$c_{Z, N}: Z \otimes N \to N \otimes Z$ by factoring through the canonical projections $Z \otimes N \to Z
\otimes_A N$, $N \otimes Z \to N \otimes_A Z$.

\medbreak 
In what follows $\A = (A, \sigma)$ will denote the induced central algebra of a perfect abelian exact sequence.
By Proposition \ref{char-sym}, $\A \otimes \A$ belongs to the subcategory $\dys \Z(\C)_\A$ of dyslectic $\A$-modules in $\Z(\C)$. Moreover, $\A \otimes \A = F_\A(\A)$ is a symmetric  algebra in $\dys \Z(\C)_\A$.

We can now conclude:

\begin{proposition}\label{def-b} The pair $\Bb = (A \otimes A, \tilde{c})$ is a symmetric commutative algebra in $\Z(\C_A)$,  where $c = c_{A\otimes A, - } =  (\sigma \otimes \id_A)(\id_A \otimes \sigma)$. \qed
\end{proposition}

We quote the following lemma that will be needed later on; compare with the argument in pp. 343 of \cite{schauenburg}.

\begin{lemma}\label{comm-tildec} Let $Y \in \C_A$. Then the following diagram is commutative:
$$\xymatrix{
& A \otimes Y \ar[r]^{\sigma_Y}\ar[d]_{\id_A \otimes \eta \otimes \id_Y}& Y \otimes A \ar[d]^{\id_Y \otimes \id_A \otimes \eta}\\
	& A \otimes A \otimes Y \ar[d]_{} & Y \otimes A \otimes A \ar[d]^{} \\
	& A \otimes A \otimes_A Y \ar[r]_{\tilde c_Y}& Y \otimes_A A \otimes A.}$$
\end{lemma}

\begin{proof} Since $\eta$ is a morphism in $\Z(\C)$, we have
\begin{align*}
c_Y \, (\id_A \otimes \eta \otimes \id_Y) & = (\sigma_Y \otimes \id_A) \, (\id_A \otimes \sigma_Y) \, (\id_A \otimes \eta \otimes \id_Y) \\
& = (\sigma_Y \otimes \id_A) \, (\id_A \otimes \id_Y \otimes \eta) \\
& = (\id_Y \otimes \id_A \otimes \eta) \, \sigma_Y.
\end{align*}
This shows that the upper square in the following diagram commutes:
$$\xymatrix{
& A \otimes Y \ar[r]^{\sigma_Y}\ar[d]_{\id_A \otimes \eta \otimes \id_Y}& Y \otimes A \ar[d]^{\id_Y \otimes \id_A \otimes \eta}\\
& A \otimes A \otimes Y \ar[r]^{c_Y} \ar[d]_{} & Y \otimes A \otimes A \ar[d]^{} \\
& A \otimes A \otimes_A Y \ar[r]_{\tilde c_Y}& Y \otimes_A A \otimes A.}$$
On the other hand, the bottom square commutes by definition of $\tilde c$.
Hence the outside rectangle commutes, which proves the lemma. 	
\end{proof}

\section{Abelian extensions $\Rep G \toto \C \toto \D$}\label{ab-ext}

In this section we consider an abelian exact sequence of finite tensor categories  $(\E): \Rep G \toto \C \overset{F}\toto \D$. In what follows we shall identify $(\E)$  with the equivalent exact sequence
\begin{equation}\label{ab-tsec}
\langle A \rangle \toto \C \overset{F_A}\toto \C_A,
\end{equation}
where $(A, \sigma)$ is the induced central algebra.

\subsection{$G$-action on $\C_A$}\label{ss-action}

Recall that $(A, \sigma\vert_{\langle A\rangle})$ is canonically isomorphic to the induced central algebra of the tensor functor $\Hom_{\D}(\1, Ff(-)): \Rep G \to \vect$. By Proposition \ref{ind-forg} the functor $\Hom_{\D}(\1, Ff(-))$ is monoidally isomorphic to the forgetful functor $U: \Rep G \to \vect$. Then, as pointed out in Remark \ref{rmk-g}, there are group isomorphisms
\begin{equation}\label{iso-g}G \cong \Aut_{\otimes}(U) \cong \Aut_\C(A).
\end{equation}
In the sequel we shall indicate by $\alpha_g \in \Aut_\C(A)$, $g \in G$, the automorphism of $A$ corresponding to an element $g \in G$ under the isomorphisms \eqref{iso-g}. Then we have $\alpha_{gh} = \alpha_g \alpha_h$, for all $g, h \in G$.

\medbreak
For every  $g \in G$, let $\rho^g: \C_A \to \C_A$ be the functor defined on an object $\mu_X: X \otimes A \to X$ of $\C_A$ in the form $\rho^g(X) = X$ with right $A$-module structure $$\mu^g_X = \mu_X\; (\id_X \otimes \alpha_g): X \otimes A \to X,$$ and $\rho^g(f) = f$ on a morphism $f$ in $\C_A$. We obtain in this way a strict right action of $G$ on $\C_A$ by $k$-linear autoequivalences, that is, a strict monoidal functor
\begin{equation}\rho: \underline G^{\op} \to \Aut_k (\C_A).
\end{equation}


\begin{remark}
The functors $\rho^g: \C_A \to \C_A$, $g \in G$, are not tensor autoequivalences in general. However, they are $\C$-module functors with respect to the canonical $\C$-module category structure on $\C_A$.
Let $g \in G$. Then $\rho^g(X \ootimes Y) = X \ootimes \rho^g(Y)$, for all $X \in \C$, $Y \in \C_A$. Moreover, the identity morphism $\rho^g \circ \ootimes \to \ootimes \circ (\id_\C \times \rho^g)$ makes $\rho^g: \C_A \to \C_A$ into a $\C$-module functor.

\medbreak Note that $\C_A$ is equivalent to the $\C$-module category $\D$ with action $X \ootimes Y = F(X)\otimes Y$, $X \in \C$, $Y \in \D$. In particular $\C_A$ is an exact indecomposable $\C$-module category.
\end{remark}

\begin{lemma}\label{ff-rho} Let $g \in G$ and suppose that there exists a nonzero object $Y \in \C_A$ such that $\mu^g_Y = \mu_Y$. Then $g = e$.
\end{lemma}

\begin{proof} The assumption implies that $\mu_{X \ootimes Y}^g = \mu_{X \ootimes Y}$, for all objects $X \in \C$. Let now $\tilde Y$ be any object of $\C_A$ with $A$-action $\mu_{\tilde Y} : \tilde Y \otimes A \to A$. Since $\C_A$ is an exact  indecomposable  $\C$-module category, there exist an object $X \in \C$ and a monomorphism of right $A$-modules $i: \tilde Y \to X \ootimes Y$; see \cite[Proposition 2.6(i)]{eg-emc}.
Then
\begin{align*} i\mu_{\tilde Y} & = \mu_{X \ootimes Y} (i \otimes \id_A)\\
& = \mu_{X \ootimes Y}^g (i \otimes \id_A) \\
& = \mu_{X \ootimes Y} (i \otimes \id_A) (\id_{\tilde Y} \otimes \alpha_g)\\
& = i\mu_{\tilde Y} (\id_{\tilde Y} \otimes \alpha_g)\\
& = i\mu_{\tilde Y}^g,
\end{align*}
and thus $\mu_{\tilde Y}^g = \mu_{\tilde Y}$, since $i$ a monomorphism.
In particular, taking $\tilde Y = A$ with right $A$-module structure given by the multiplication map $m: A \otimes A \to A$, we get $m = m^g$. Then
$\id_A = m(\eta \otimes \id_A) = m^g (\eta \otimes \id_A) = m (\eta \otimes \id_A) \alpha_g = \alpha_g$.
Hence $g = e$, as claimed.
\end{proof}

\subsection{$\Gamma$-grading on $\C_A$}\label{ss-grading}
Recall from Proposition \ref{def-b} that  $\Bb = (A \otimes A, \tilde{c})$ is a commutative algebra in $\Z(\C_A)$, where $c = c_{A\otimes A, - } =  (\sigma \otimes \id_A)(\id_A \otimes \sigma)$.

\medbreak
We have $A \otimes A = \U(\Bb)$, where $\U: \Z(\C_\A) \to \C_A$ is the forgetful functor. Note that $A \otimes A = F_A(A)$ is an algebra in $\C_A$ with multiplication $m_{A\otimes A}: A \otimes A \otimes_A A \otimes A \to A\otimes A$ induced by the multiplication of $A$ in the form
$m_{A \otimes A} = (m \otimes \id_A)(F_A^2)_{A, A}$. That is, $m_{A\otimes A}$ is obtained from the morphism
\begin{equation}(m  \otimes m) (\id_A \otimes \sigma \otimes \id_A):
A \otimes A \otimes A \otimes A \to A \otimes A,
\end{equation}
by factoring through the epimorphism $A \otimes A \otimes A \otimes A \to A \otimes A \otimes_A A \otimes A$.

\medbreak
Let $(Z, c) \in \Z(\C_A)$ and suppose that $\U(Z, c) = Z$ is a trivial object of $\C_A$. Then $(Z, c)$ is endowed with a \emph{trivial half-braiding} $\tau: - \otimes_A Z \to Z \otimes_A -$, that is, a natural isomorphism $$\tau_{X}: Z \otimes_A X \to X \otimes_A Z,\quad X \in \C_A.$$
Let $X$ be a simple object  of $\C_A$. Then $X$ gives rise to a natural automorphism of $Z\otimes_A X$ in the form
\begin{equation}\label{autx}
Z \otimes_A X \overset{\tau_{X}}\toto X \otimes_A Z \overset{c_{X}^{-1}}\toto Z \otimes_A X,\end{equation}
Since $X$ is  simple and $Z$ is a trivial object of $\C_A$, then  \eqref{autx} corresponds to an automorphism $\partial(X)_{(Z,c)}$ of $Z$ such that
the following diagram is commutative:
\begin{equation}\label{diag-partial}\xymatrix@C=0.75in @R=0.45in{Z \otimes_A X \ar[r]^{\partial(X)_{(Z, c)}\otimes_A \id_X} \ar[rd]_{\tau_{X}} & Z \otimes_A X \ar[d]^{c_{X}}\\
	& X \otimes_A Z.}
\end{equation}
Moreover, this correspondence is natural in $(Z, c) \in \KER_{\U}$. 
The commutativity of \eqref{diag-partial} implies that
\begin{equation}\label{mult}
\partial(X)_{(Z, c) \otimes (Z', c')} = \partial(X)_{(Z, c)} \otimes_A \partial(X)_{(Z', c')}, \quad (Z, c), (Z', c') \in \KER_\U.\end{equation}

\medbreak
Since $A \otimes A$ is a trivial object of $\C_A$, we obtain from the previous discussion an automorphism $\partial(X) = \partial(X)_{\Bb}$ of $A \otimes A$.
Equation \eqref{mult} combined with the naturality of $\partial(X)_{(Z, c)}$ in $(Z, c)$ and the fact that $\U$ is a strict monoidal functor imply that $\partial(X)$ is an algebra automorphism of $A \otimes A$.

\medbreak
Let $\Aut_{A}(A\otimes A)$ denote the group of algebra automorphisms of $A \otimes A$ in $\C_A$, which is a finite group.

\begin{proposition}\label{gd-tilde} The map $\partial: \Irr(\C_A) \to \Aut_{A}(A\otimes A)$ defines an $\Aut_{A}(A\otimes A)$-grading on $\C_A$.
\end{proposition}

\begin{proof} 	
Let $X_1, X_2$ be simple objects of $\C_A$. The naturality of $\tau$ combined with the half-braiding condition on $\tau$ and $c$ imply the commutativity of the following diagram
$$\xymatrix@C=1.5in @R=0.5in{A \otimes A \otimes_A X_1 \otimes_A X_2  \ar[r]^{\partial(X_1)\partial(X_2) \otimes_A \id_{X_1 \otimes_A X_2}} \ar[rd]_{\tau_{A \otimes A, X_1\otimes_A X_2}\; } & A \otimes A \otimes_A X_1 \otimes_A X_2 \ar[d]^{c_{A \otimes A, X_1 \otimes_A X_2}}\\
 & X_1 \otimes_A X_2 \otimes_A A \otimes A.}$$
Then $\partial (X_0) = \partial (X_1)\partial (X_2)$, for every composition factor $X_0$ of $X_1\otimes_A X_2$. This implies the proposition. \end{proof}

In what follows we shall denote by $\Gamma \subseteq \Aut_{A}(A\otimes A)$ the \emph{support} of $\C_A$ with respect to the grading in Proposition \ref{gd-tilde}. That is, $\Gamma = \{s \in \tilde \Gamma:\; (\C_A)_s \neq 0 \}$. Thus $\C_A$ is faithfully graded by the group $\Gamma$.

\subsection{Permutation action $\triangleright: \Gamma \times G \to G$.}\label{def-rhd}

Recall that the forgetful functor $\C_A \to \C$ is right adjoint of the functor $F_A : \C \to \C_A$. For every $X \in \C$, $Y \in \C_A$, the natural isomorphism $\phi_{X, Y}: \Hom_\C(X, Y) \to \Hom_A(X\otimes A, Y)$ is given by
\begin{equation*} \phi_{X, Y}(T) = \mu_Y (\alpha \otimes \id_A), \quad  \alpha \in \Hom_\C(X, Y),
\end{equation*}
whose inverse $\psi_{X, Y}: \Hom_A(X\otimes A, Y) \to \Hom_\C(X, Y)$ is defined as
\begin{equation*}\psi_{X, Y}(b) = b (\id_X \otimes \eta), \quad b \in \Hom_A(X\otimes A, Y).
\end{equation*}

\medbreak
Thus $\Hom_A(A\otimes A, A)$ is a $k$-algebra with multiplication $*$ induced by composition in $\Hom_\C(A, A)$, that is,
\begin{equation}\label{mult-alg}
b * b' = m \left(b(\id_A \otimes \eta) \; b' (\id_A \otimes \eta) \otimes \id_A\right), \quad b, b' \in \Hom_A(A\otimes A, A).
\end{equation}

\begin{lemma}\label{gp-iso} The adjunction isomorphism $\psi_{A, A}: \Hom_A(A\otimes A, A) \to \Hom_\C(A, A)$ induces a bijection $\operatorname{Alg}_A(A\otimes A, A) \cong \Aut_\C(A)$.
\end{lemma}

In particular, $\operatorname{Alg}_A(A\otimes A, A)$ is a group with multiplication defined by \eqref{mult-alg}.

\begin{proof} It is straightforward to verify that a morphism $b: A\otimes A \to  A$ is an algebra morphism in $\C_A$ if and only if $\psi_{A, A}(b): A \to A$ is an algebra morphism in $\C$. This implies the lemma since, as observed in Subsection \ref{repg}, every algebra morphism of $A$ in $\C$ is an automorphism.
\end{proof}

Let $g \in G$. We shall denote by $p_g = \phi(\alpha_g): A \otimes A \to A$ the morphism in $\C_A$ corresponding to $\alpha_g \in \Aut_{\C}(A)$. Thus we have a commutative diagram in $\C$:
\begin{equation}\xymatrix@C=0.75in @R=0.45in{A \ar[r]^{\id_A \otimes \eta \quad } \ar[rd]_{\alpha_g} & A \otimes A \ar[d]^{p_g}\\ & A.}
\end{equation}

Notice that the group $\Gamma \subseteq \Aut_A(A\otimes A)$ acts on $\operatorname{Alg}_A(A\otimes A, A)$ in the form
\begin{equation}
s . b = b s^{-1}: A \otimes A \to A, \qquad s\in \Gamma, b \in \operatorname{Alg}_A(A\otimes A, A).
\end{equation}

Let $g \in G$. It follows from Lemma \ref{gp-iso}, that for all $s \in \Gamma$, there exists a unique element of $G$, that we denote $s \rhd g$, such that $s . p_g = p_{s \rhd g}$. This gives rise to an action by permutations
\begin{equation} \rhd: \Gamma \times G \to G,
\end{equation}
which is determined by the commutativity of the diagrams
\begin{equation}
\xymatrix{A \otimes A \ar[r]^{s} \ar[rd]_{p_{s^{-1} \rhd g}} & A \otimes A \ar[d]^{p_g}\\ & A,}
\end{equation}
for all $s \in \Gamma$, $g \in G$.

\medbreak Observe that the group $\Aut_{\Z(\C)}(A, \sigma)$ can be regarded as a subgroup of $\Aut_{\C}(A) \cong G$ in a canonical way. The action $\rhd: \Gamma \times G \to G$ gives some control on the obstruction to have an equality $\Aut_{\Z(\C)}(A, \sigma) = \Aut_{\C}(A)$. More precisely, we have:

\begin{proposition}\label{a-s} Let $g \in G$ and $s \in \Gamma$. Then for every object $Y \in (\C_A)_s$ the following diagram is commutative:
\begin{equation}\label{alfa-sigma}
\xymatrix@C=0.6in @R=0.45in{
A \otimes Y \ar[d]_{\alpha_{s\rhd g} \otimes \id_Y} \ar[r]^{\sigma_Y}
& Y\otimes A \ar[r]^{\id_Y \otimes \alpha_g}  & Y \otimes A \ar[d]^{\mu_Y}\\
A \otimes Y \ar[r]_{\sigma_Y} & Y \otimes A \ar[r]_{\mu_Y} & Y.}\end{equation}
\end{proposition}

\begin{proof} Recall that $\Bb = \A \otimes \A = (A \otimes A, \tilde c) \in \Z(\C_A)$, where $c = c_{A \otimes A, -} = (\sigma \otimes \id_A)(\id_A \otimes \sigma)$. Let $s \in \Gamma$ and let $Y \in (\C_A)_s$. By definition of the grading of $\C_A$ given in Subsection \ref{ss-grading}, the half-braiding $\tilde c_Y : A \otimes A \otimes_A Y \to Y \otimes_A A \otimes A$ satisfies $\tilde c_Y  \; \tau_{Y, A\otimes A}= \id_Y \otimes_A \; s^{-1}$, where $\tau$ is the flip isomorphism. Hence the triangle on the left of the following diagram
$$
\xymatrix@R=0.7em @C=5em{
	&A\otimes A \otimes_A Y \ar[r]^{p_g \otimes \id_Y} \ar[dd]^{\tilde c_Y}  & A\otimes_A Y  \ar[dd]^{\tau_{A, Y} }\\
	Y\otimes_A A \otimes A \ar[ur]^-{\tau_{Y, A\otimes A}} \ar[dr]_-{\id_Y \otimes s^{-1}} &&\\
	& Y\otimes_A A \otimes A \ar[r]_{\quad \id_Y \otimes p_{s^{-1}\rhd g}}& Y \otimes_A
	A, }
$$	
commutes. By definition of the action $\rhd$, the composition in the bottom of the diagram equals $\id_Y \otimes p_g : Y\otimes_A A \otimes A \to Y \otimes_A A$, which by naturality of $\tau$ coincides with the composition
$\tau_{A, Y} \, (p_g \otimes \id_Y) \, \tau_{Y, A \otimes A}$.
Hence the square on the right of the diagram commutes.

\medbreak
Consider the following diagram of morphisms in $\C$:

\bigskip
$$
\xymatrix@R=2em @C=5em{
&	& A\otimes A \otimes Y \ar[r]^{p_g \otimes \id_Y } \ar[d]^{}  & A\otimes Y  \ar[d]_{}  \ar[r]^{\sigma_Y}& Y \otimes A \ar[d]^{\mu_Y}\\
& A \otimes Y \ar @/^4.5pc/@{->}[urr]^{\alpha_g \otimes \id_Y} \ar[r]^{\cong \quad } \ar[d]_{\sigma_{Y}} \ar[ur]^{\id_{\!A}\!\otimes \eta \otimes \id_Y} & A\otimes A \otimes_A Y \ar[r]^{p_g \otimes_A \id_Y} \ar[d]^{\tilde c_{Y}} & A\otimes_A Y \ar[d]^{\tau_{A, Y}} \ar[r]^{\cong} & Y\\
& Y \otimes A \ar[r]_{\cong \quad } \ar[dr]_{\id_Y\!\otimes\!\id_A\otimes \eta} \ar @/_4.5pc/@{->}[drr]_{\id_Y \otimes \alpha_{s^{-1}\rhd g}} & Y \otimes_A A \otimes A \ar[r]_{\quad \id_Y\!\otimes_A p_{s^{-1}\rhd g}\; }  & Y\otimes_A A \ar[r]^{\cong}& Y\ar[u]_{=}\\
&	& Y\otimes A \otimes A \ar[r]_{\quad \id_Y\!\otimes p_{s^{-1}\rhd g}\quad}\ar[u]^{} & Y \otimes A \ar[u]^{} \ar[ur]_{\mu_Y}& }
$$

\bigskip
\noindent where the unspecified arrows are all canonical. As seen before, the middle square within the  two middle rows is commutative. Using the basic properties of the tensor product in $\C_A$, we obtain that the squares over it and under it, the two triangles on the left-most part of the diagram and the three diagrams on the right-most part of the diagram are commutative.
The curved triangles on top and bottom of the diagram are commutative by definition of $p_g$, $g \in G$.

Finally,  the left-most square within the two middle rows is commutative by Lemma \ref{comm-tildec}.
Therefore we obtain that the outside diagram is commutative. This implies the commutativity of the diagram \eqref{alfa-sigma} and finishes the proof of the lemma.
\end{proof}

\begin{remark}Let $\lambda_Y = \mu_Y \sigma_Y: A \otimes Y \to Y \otimes A$ be the left action corresponding to $\mu_Y$. The statement in Proposition \ref{a-s} can be reformulated by saying that $\lambda_Y^{s\rhd g} = \lambda_{\rho^g(Y)}$, for all $Y \in (\C_A)_s$.
\end{remark}

\begin{corollary}\label{exists-gama} Let $g \in G$, $t \in \Gamma$. For every $X \in \C_A$, $Y \in (\C_A)_t$, the identity morphism $X \otimes Y \to X \otimes Y$ induces a canonical isomorphism $\rho^g(X \otimes_A Y) \cong \rho^{t\rhd g}(X) \otimes_A \rho^g(Y)$ of right $A$-modules in $\C$.
\end{corollary}

\begin{proof}
As an object of $\C$, $\rho^g(X \otimes_A Y)$ is the coequalizer of the diagram
\begin{equation*}
\xymatrix@C=5.5em{ X \otimes A \otimes Y \ar[r]<1ex>^{\mu_X \otimes \id_Y} \ar@<-1ex>[r]_{\quad \id_X \otimes \mu_Y\sigma_Y}  & X \otimes Y.}
\end{equation*}	
Similarly, $\rho^{t\rhd g}(X) \otimes_A \rho^g(Y)$ is the coequalizer of the diagram
\begin{equation*}
\xymatrix@C=5.5em{ X \otimes A \otimes Y \ar[r]<1ex>^{\mu_X^{t\rhd g} \otimes \id_Y} \ar@<-1ex>[r]_{\quad \id_X \otimes \mu_Y^g\sigma_Y} & X \otimes Y.}
\end{equation*}	
On the one hand, we have
$$\mu_X^{t\rhd g} \otimes \id_Y = (\mu_X \otimes \id_Y) \; (\id_X \otimes \alpha_{t\rhd g} \otimes \id_Y),$$	
and on the other hand, using Proposition \ref{a-s},
\begin{align*}\id_X \otimes \mu_Y^g\sigma_Y & = (\id_X \otimes \mu_Y)\; (\id_X \otimes \id_Y \otimes \alpha_g) \; (\id_X \otimes \sigma_Y) \\
& = (\id_X \otimes \mu_Y)\; (\id_X \otimes \sigma_Y) \; (\id_X \otimes \alpha_{t\rhd g} \otimes \id_Y) \\
& = (\id_X \otimes \mu_Y\sigma_Y)\;  (\id_X \otimes \alpha_{t\rhd g} \otimes \id_Y).
\end{align*}
Thus the identity morphism $X \otimes Y \to X \otimes Y$ induces a canonical isomorphism $\rho^g(X \otimes_A Y) \cong \rho^{t\rhd g}(X) \otimes_A \rho^g(Y)$ in $\C$. It is clear that this isomorphism commutes with the right $A$-module structures. This proves the statement.
\end{proof}

\subsection{Permutation action $\lhd:\Gamma \times G \to \Gamma$}\label{def-lhd}
We show in this subsection that, for every $g \in G$, the $k$-linear autoequivalence  $\rho^g:\C \to \C$ permutes the homogeneous components $\C_s$, $s \in \Gamma$, thus giving rise to a permutation action $\lhd:\Gamma \times G \to \Gamma$.

\begin{lemma}\label{ce-stable} The trivial homogeneous component $(\C_A)_e$ is stable under the endofunctors $\rho^g$, for all $g \in G$.
\end{lemma}

\begin{proof} Let $X$ be an object of $(\C_A)_e$. Then $\tilde c_{X} = \tau_{A \otimes A, X}: A\otimes A \otimes_A X \to X \otimes_A A \otimes A$, where $\tau$ is the trivial half-braiding. In particular, $\tilde c_{A\otimes A} = \tau_{A \otimes A, A \otimes A}$, because $A \otimes A$ is a trivial object of $\C_A$.

On the other hand, $X \otimes A \cong X \otimes_A A \otimes A$ in $\C_A$ and since $\tilde c: A \otimes A \otimes_A - \to - \otimes_A A \otimes A$ is a half-brading, there is a commutative diagram:
$$
\xymatrix@R=5em @C=0.3em{A \otimes A \otimes_A X \otimes_A A \otimes A \ar[dr]_{\tau_X \otimes_A \id_{A \otimes A}\quad } \ar[rr]^{\tilde c_{X \otimes_A A \otimes A}} & & X \otimes_A A \otimes A \otimes_A A \otimes A \\
 & X \otimes_A A \otimes A \otimes_A A \otimes A \ar[ur]_{\qquad \id_X \otimes_A \tau_{A\otimes A, A \otimes A}}. & }
$$
Therefore $\tilde c_{X \otimes A} =  \tilde c_{X \otimes_A A \otimes A} = \tau_{A \otimes A, X \otimes A}$ and thus $X \otimes A \in (\C_A)_e$.

\medbreak
Note that $\eta \otimes \id_A: A \to A \otimes A$ is a split monomorphism in $\C_A$. Hence for every object $X \in \C_A$, $X \cong X \otimes_A A$ is a direct summand of $X \otimes_A A \otimes A \cong X \otimes A$. Let now $g \in G$. Then $\rho^g(X)$ is a direct summand of $\rho^g(X \otimes A)$. In addition $\id_X \otimes \alpha_g: X \otimes A \to \rho^g(X \otimes A)$ is an isomorphism of right $A$-modules in $\C$. This implies that $\rho^g(X)$ is a direct summand of $X \otimes A$. Thus $\rho^g(X) \in (\C_A)_e$ for all $g \in G$, as claimed.
\end{proof}

\begin{proposition}\label{compatible} There exists an action by permutations  $\lhd:\Gamma \times G \to \Gamma$ such that $e \lhd g = e$ and $$\rho^g((\C_A)_s) \subseteq (\C_A)_{s \lhd g},$$ for all $g \in G$, $s \in \Gamma$. \end{proposition}

\begin{proof} Let $s \in \Gamma$ and let $X$ be a simple object of $(\C_A)_s$. For every $g \in G$, $\rho^g(X)$ is a simple object of $\C_A$ and therefore there exists $s' \in \Gamma$ such that $\rho^g(X) \in (\C_A)_{s'}$. Suppose that $Y \in (\C_A)_s$ is another simple object and $\rho^g(Y) \in (\C_A)_{s''}$, $s'' \in \Gamma$.

\medbreak
Then $X^* \otimes_A Y \in (\C_A)_{s^{-1}s} = (\C_A)_{e}$ and by Lemma \ref{ce-stable}, $\rho^g(X^* \otimes_A Y) \in (\C_A)_e$. It follows from Corollary  \ref{exists-gama} that $\rho^g(X^* \otimes_A Y) \cong \rho^{s \rhd g}(X^*) \otimes_A \rho^g(Y)$. Since $X^*$ is also a simple object, then $\rho^{s\rhd g}(X^*) \in (\C_A)_t$, for some $t \in \Gamma$, whence
$\rho^g(X^* \otimes_A Y) \in (\C_A)_{ts''}$. This entails that $s'' = t^{-1}$.
Since $Y \in (\C_A)_s$ was an arbitrary simple object, we can conclude that $s' = t^{-1} = s''$.

\medbreak
Therefore for every $g \in G$, $s \in \Gamma$, there is a well-defined element $s \lhd g \in \Gamma$ such that  $\rho^g(\C_s) \subseteq \C_{s\lhd g}$. Since $\rho^g\rho^h = \rho^{hg}$, for all $g, h \in G$, this gives rise to an action by permutations $\lhd: \Gamma  \times G \to \Gamma$ that satisfies the required property.
\end{proof}

\subsection{Proof of Theorem \ref{main1}}\label{proof-main1} Suppose that $(\E): \; \Rep G \toto \C \toto \D$ is an abelian exact sequence of finite tensor categories. We may identify $(\E)$ with the equivalent an exact sequence \eqref{ab-tsec},
where $(A, \sigma)$ is the induced central algebra of $(\E)$.

\medbreak
Let $\Gamma \subseteq \Aut_A(A \otimes A)$ be the subgroup introduced in Subsection \ref{ss-grading}.
Let also $\rhd: \Gamma \times G \to G$ and $\lhd: \Gamma \times G \to \Gamma$ be the permutation actions defined in Subsections \ref{def-rhd} and \ref{def-lhd}, respectively. Theorem \ref{main1} will follow from the next result:

\begin{theorem}\label{main-ca} The permutation actions $\rhd$ and $\lhd$ make $(G, \Gamma)$ into a matched pair of finite groups. Furthermore, $\C_A$ is endowed with a $(G, \Gamma)$-crossed action such that $\C \cong (\C_A)^{(G, \Gamma)}$ as extensions of $\C_A$ by $\Rep G$.
\end{theorem}

\begin{proof} Let $s, t \in \Gamma$ and let $g, h \in G$. Since the $\Gamma$-grading on $\C_A$ is faithful we may pick nonzero objects $X \in (\C_A)_s$ and $Y \in (\C_A)_t$. Then $X \otimes_A Y \in (\C_A)_{st}$ and thus $\rho^g(X \otimes_A Y) \in (\C_A)_{st \lhd g}$, by Proposition \ref{compatible}. On the other hand, combining  Proposition \ref{compatible} and Corollary \ref{exists-gama}, we find that $\rho^g(X \otimes_A Y) \cong \rho^{t\rhd g}(X) \otimes_A \rho^g(Y) \in (\C_A)_{(s \lhd (t\rhd g)) \; (t \lhd g)}$. Therefore $st \lhd g = (s \lhd (t\rhd g)) \; (t \lhd g)$.
	
\medbreak
Since $Y \in (\C_A)_t$, it follows from Proposition \ref{a-s} that for all $x \in G$,
\begin{equation}\label{allx} \mu_Y^x \sigma_Y = \mu_Y \; (\id_Y \otimes \alpha_x) \; \sigma_Y = \mu_Y \; \sigma_Y \; (\alpha_{t \rhd x} \otimes \id_Y),
\end{equation}
Therefore
$$\mu_Y \; (\id_Y \otimes \alpha_{gh}) \sigma_Y  = \mu_Y \sigma_Y \; (\alpha_{t \rhd gh} \otimes \id_Y).$$
On the other hand,
\begin{align*}
\mu_Y \; (\id_Y \otimes \alpha_{gh}) \sigma_Y  & = \mu_Y^g \; (\id_Y \otimes \alpha_{h}) \; \sigma_Y \\
& = \mu_Y^g \; \sigma_Y \; (\alpha_{(t\lhd g)\rhd h} \otimes \id_Y)\\
& = \mu_Y \; \sigma_Y \; (\alpha_{t \rhd g} \alpha_{(t\lhd g)\rhd h} \otimes \id_Y),
\end{align*}
the second equality because $\rho^g(Y) \in (\C_A)_{t \lhd g}$ and the third by \eqref{allx}.

Consider the left $A$-module structure $\lambda_Y = \mu_Y\sigma_Y: A \otimes Y \to Y$. The previous discussion shows that $$\lambda_Y^{t \rhd gh} = \lambda_Y^{(t \rhd g) \; ((t\lhd g)\rhd h)}.$$ Hence the left-hand version of Lemma \ref{ff-rho} implies that $t\rhd gh = (t\rhd g) \; ((t \lhd g)\rhd h)$. Therefore $(G, \Gamma)$ is a matched pair, as claimed.

\medbreak The right $k$-linear action $\rho: \underline{G}^{op} \to \Aut_k(\C_A)$ together with the $\Gamma$-grading $\C_A = \bigoplus_{s \in \Gamma}(\C_A)_s$ from Subsections \ref{ss-action} and \ref{ss-grading} give rise to a $(G, \Gamma)$-crossed action on $\C_A$ as follows.

Let $g \in G$ and $t \in \Gamma$. Let also $X \in \C_A$, $Y \in (\C_A)_t$. By Corollary \ref{exists-gama}, we may identify canonically $\rho^g(X \otimes_A Y)$ and $\rho^{t\rhd g}(X) \otimes_A \rho^g(Y)$ as right $A$-modules in $\C$.
Under this identification we let, for every $g \in G$, $\gamma^g_{X, Y}$ to be the identity morphism $\rho^g(X \otimes_A Y) \to \rho^{t\rhd g}(X) \otimes_A \rho^g(Y)$.
Let also $\gamma_0^g = \alpha_g^{-1}: \rho^g(A) \to A$. Since  $\rho_0$, $\rho_2$, $\rho^g_2$ are identities, then we obtain the commutativity of the diagrams (a)--(e) in Subsection \ref{ss-crossed}.

\medbreak
As a $k$-linear category, the corresponding extension $(\C_A)^{(G, \Gamma)}$ coincides with $(\C_A)^G$ and the canonical functors give rise to an exact sequence of tensor categories
$$\Rep G \toto (\C_A)^{(G, \Gamma)} \overset{F}\toto \C_A.$$
Let $\varphi: \C \to \C^{(G, \Gamma)}$ be the $k$-linear functor defined in the form $$\varphi(X) = (X \otimes A, (\id_X \otimes \alpha_g^{-1})_{g \in G}), \quad X \in \C.$$

We claim that $\varphi$ is a tensor functor with monoidal isomorphisms $\varphi^2_{X, Y} = (F_A^2)_{X, Y}: \varphi(X) \otimes_A \varphi(Y) \to \varphi(X \otimes Y)$, and $\varphi^0 = F_A^0: A \to \varphi(\1)$. To prove this claim, we only need to show that $(F_A^2)_{X, Y}$ is a morphism in $(\C_A)^{(G, \Gamma)}$.

We have $\varphi(X \otimes Y) = (X \otimes Y \otimes A, (\id_{X \otimes Y} \otimes \alpha_g^{-1})_{g \in G})$ and  from the definition of the tensor product in $(\C_A)^{(G, \Gamma)}$ (c.f. Subsection \ref{ss-crossed}),
$$\varphi(X) \otimes \varphi(Y) = \left(X \otimes A \otimes_A Y \otimes A, (\bigoplus_{s \in \Gamma}\left(\id_X \otimes \alpha_{s\rhd g}^{-1}) \otimes_A (\id_Y \otimes \alpha_g^{-1})\right)_{g\in G}\right).$$

Recall that $(F_A^2)_{X, Y}: X \otimes A \otimes_A Y \otimes A \to X \otimes Y \otimes A$ is the morphism induced by $(\id_{X \otimes Y} \otimes m) \, (\id_X \otimes \sigma_Y \otimes \id_A): X \otimes A \otimes Y \otimes A \to X \otimes Y \otimes A$.

\medbreak
Let $g \in G$. On the one hand, using that $\sigma_Y \otimes \id_A = (\id_Y \otimes \sigma_A^{-1}) \, \sigma_{Y \otimes A}$ and the fact that $A$ is commutative, we have
\begin{align*}
(\id_{X \otimes Y} \otimes \alpha_g^{-1}) & \, (\id_{X \otimes Y} \otimes m) \, (\id_X \otimes \sigma_Y \otimes \id_A)  \\
& =
(\id_{X \otimes Y} \otimes \alpha_g^{-1}) (\id_{X \otimes Y} \otimes m) \, (\id_{X \otimes Y} \otimes \sigma^{-1}) \, (\id_X \otimes \sigma_{Y\otimes A})\\
& = (\id_{X \otimes Y} \otimes \alpha_g^{-1}m) \, (\id_X \otimes \sigma_{Y\otimes A}).
\end{align*}

On the other hand, let $Y \otimes A = \oplus_{s \in \Gamma} (Y \otimes A)_s$ be a decomposition of $Y\otimes A$ as a direct sum of homogeneous objects $(Y \otimes A)_s \in (\C_A)_s$. Then
\begin{align*}
(\id_{X \otimes Y}& \otimes m)  \, (\id_X \otimes \sigma_Y \otimes \id_A) \, \bigoplus_{s \in \Gamma}\left(\id_X \otimes \alpha_{s\rhd g}^{-1} \otimes \id_Y \otimes \alpha_g^{-1}\right)  \\
& = (\id_{X \otimes Y} \otimes m)  \, (\id_X \otimes \sigma_{Y \otimes A}) \, \bigoplus_{s \in \Gamma}\left(\id_X \otimes \alpha_{s\rhd g}^{-1} \otimes \id_{\rho^g((Y \otimes A)_s)}\right) \, (\id_{X \otimes A \otimes Y}  \otimes \alpha_g^{-1}) \\
& = (\id_{X \otimes Y} \otimes m) \,  \bigoplus_{s \in \Gamma}\left(\id_X \otimes \sigma_{\rho^g((Y \otimes A)_s)}(\alpha_{s\rhd g}^{-1} \otimes \id_{\rho^g(Y \otimes A)_s)})   \right)\, (\id_{X \otimes A \otimes Y}  \otimes \alpha_g^{-1}) \\
& = \bigoplus_{s \in \Gamma}\left(\id_X \otimes \mu_{\rho^g((Y \otimes A)_s)}\sigma_{\rho^g((Y \otimes A)_s)}(\alpha_{(s\rhd g)^{-1}} \otimes \id_{\rho^g((Y \otimes A)_s)})\right) \\
& = \bigoplus_{s \in \Gamma}\left(\id_X \otimes \mu_{\rho^g((Y \otimes A)_s)} (\id_{\rho^g((Y \otimes A)_s)} \otimes \alpha_{g^{-1}}) \sigma_{\rho^g((Y \otimes A)_s)}\right) \\
& = (\id_{X \otimes Y} \otimes m) \, (\id_X \otimes \id_Y \otimes \alpha_g^{-1}  \otimes \id_A) \, \bigoplus_{s \in \Gamma}\left(\id_X \otimes (\id_{(Y \otimes A)_s} \otimes \alpha_g^{-1}) \sigma_{(Y \otimes A)_s}\right) \\
& = (\id_{X \otimes Y} \otimes m)  \, (\id_X \otimes \id_Y \otimes \alpha_g^{-1}  \otimes \alpha_g^{-1}) \, (\id_X \otimes \sigma_{Y \otimes A})\\
& = (\id_{X \otimes Y} \otimes \alpha_g^{-1} m)  \,  (\id_X \otimes \sigma_{Y \otimes A}).
\end{align*}
The fourth equality by Proposition \ref{a-s}, since for all $s\in \Gamma$ $\rho^g((Y\otimes A)_s)$ is homogeneous of degree $s\lhd g$ and $(s \rhd g)^{-1} = (s\lhd g) \rhd g^{-1}$.

Thus we obtain that $$(\id_{X \otimes Y} \otimes \alpha_g^{-1}) \, \rho^g((F_A^2)_{X, Y}) = \left((\id_X \otimes \alpha_{s\rhd g}^{-1}) \otimes_A (\id_Y \otimes_A {\alpha_g}^{-1})\right) \, (F_A^2)_{X, Y},$$ for all $g \in G$, that is, $(F_A^2)_{X, Y}$ is a morphism in $(\C_A)^{(G, \Gamma)}$, as claimed.

\medbreak
Moreover, since $\Rep G \cong \langle \1_{\C_A}\rangle^G$ is an equivariantization, $\varphi$ induces by restriction an equivalence of tensor categories $\langle A \rangle \to \langle \1_{\C_A}\rangle^G$. Then we get a commutative diagram of tensor functors
$$\begin{CD}\langle A \rangle @>>> \C @>{F_A}>> \C_A \\
@VV{\varphi}V @VV{\varphi}V @VV{=}V\\
\langle \1_{\C_A}\rangle^G @>>> (\C_A)^{(G, \Gamma)} @>{F}>> \C_A.
\end{CD}$$
This implies that $\varphi$ is an equivalence of tensor categories and finishes the proof of the theorem.
\end{proof}

Recall that a fusion category $\C$ is called \emph{weakly group-theoretical} if $\C$ is categorically Morita equivalent to a nilpotent fusion category  \cite{ENO2}.

\begin{corollary}\label{wgt-ext-rep} Let  $\Rep G \toto \C \toto \D$ be an abelian exact sequence of finite tensor categories. Then $\C$ is Morita equivalent to a $G \bowtie \Gamma$-graded extension of a tensor subcategory $\D_0$ of $\D$. In particular, $\C$ is a weakly group-theoretical fusion category if and only if $\D$ is a weakly group-theoretical fusion category.
\end{corollary}

\begin{proof} The first statement follows from Theorems \ref{main-ca} and \ref{grading-crossed}. The second statement from Corollary \ref{wgt-crossed}.
\end{proof}

\begin{corollary} Let  $\Rep G \toto \C \toto \D$ be an abelian exact sequence of finite tensor categories. Assume in addition that the universal grading group of $\D$ is trivial. Then \eqref{ext-rep} is an equivariantization exact sequence. In other words, $\C \cong \D^G$ as extensions of $\D$ by $\Rep G$.
\end{corollary}

\begin{proof} By Theorem \ref{main-ca}, there exists a matched pair of groups $(G, \Gamma)$ such that $\C \cong \D^{(G, \Gamma)}$. On the other hand, every faithful grading of $\D$ factorizes over the universal grading. Then the assumption implies that the group $\Gamma$ is trivial. Hence the $(G, \Gamma)$-crossed action on $\C$ reduces to an action of $G$ by tensor autoequivalences and $\D^{(G, \Gamma)} \cong \D^{G}$ is the corresponding equivariantization.
\end{proof}

\section{Extensions $\vect_G \toto \C \toto \D$}\label{section-pt}

Let $G$ be a finite group.
We shall consider in this section an exact sequence of finite tensor categories
\begin{equation}\label{ext-pt}\vect_G \overset{f}\toto \C \overset{F}\toto \D.
\end{equation}

\begin{remark} It was shown in \cite[Theorem 2.9]{indp-exact} that any dominant tensor functor $F : \C \to \D$ between finite tensor categories $\C$ and $\D$ whose induced central algebra decomposes as a direct sum of invertible objects of $\C$ is in fact normal and induces an exact sequence \eqref{ext-pt}, where $G$ is the group of isomorphism classes of simple direct summands of $A$ in $\C$.
\end{remark}

\begin{lemma} Let $H$ be the induced Hopf algebra of \eqref{ext-pt}. Then $H$ is isomorphic to the group algebra $kG$.
\end{lemma}

\begin{proof} By definition of $H$, there is an equivalence of tensor categories $\comod H \cong \vect_G \cong \comod kG$. Then every simple $H$-comodule is one-dimensional and therefore $H$ is cocommutative, that is, $H \cong kG(H)$. In addition, $G(H)$ is isomorphic to the group of invertible objects of $\comod H$, whence $G(H) \cong G$. This finishes the proof of the lemma.
\end{proof}

Let $(A, \sigma)$ be the induced central algebra of \eqref{ext-pt}. Then there exist invertible objects $X_g \in \C$, $g \in G$, such that $X_g \otimes X_h \cong X_{gh}$, for all $g, h \in G$, and $A = \bigoplus_{g \in G} X_g$. The multiplication map $m: A \otimes A \to A$ is given componentwise by isomorphisms $X_g \otimes X_h \to X_{gh}$, and the half-brading $\sigma_{X_h}: A \otimes X_h \to X_h \otimes A$, $h \in G$, is given componentwise by isomorphisms
\begin{equation}\label{sigma-a}X_g \otimes X_h \to X_h \otimes X_{h^{-1}gh}, \quad g \in G.
\end{equation}
See Remark \ref{a=h}.

\medbreak Let $X$ be a simple object of $\C$. Then $X_g \otimes X$ and $X \otimes X_g$ are simple objects, for all $g \in G$. In addition, 
$A \otimes X$ and $X \otimes A$ decompose into  direct sums $A \otimes X = \bigoplus_{g \in G}X_g \otimes X$ and $X \otimes A = \bigoplus_{g \in G}X \otimes X_g$, respectively. Hence, for every $g \in G$, the isomorphism $\sigma_X: A \otimes X \to X \otimes A$ induces  an isomorphism $$\sigma^g_X: X_g \otimes X \to X \otimes X_{h^{-1}},$$ for a unique $h = \partial(X)(g) \in G$.
This gives rise to a bijection $\partial(X): G \to G$.

\medbreak
The fact that the multiplication morphism $m: A \otimes A \to A$ is a morphism in the center $\Z(\C)$ implies that $\partial(X)$ is a group automorphism of $G$.
Furthermore, we have:

\begin{proposition} The map $\partial: \Irr(\C) \to \Aut(G)$ defines an $\Aut(G)$-grading on $\C$.
\end{proposition}

\begin{proof} Let $X$ and $Y$ be simple objects of $\C$.
The relation $\sigma_{X\otimes Y} =  (\id_X \otimes \sigma_Y)(\sigma_X \otimes \id_Y)$ implies that $\sigma_{X \otimes Y}: A \otimes X\otimes Y \to X\otimes Y \otimes A$ is given componentwise by isomorphisms
\begin{equation*} X_g \otimes X \otimes Y \to X \otimes Y \otimes X_{\partial(Y)^{-1}\partial(X)^{-1}(g)}, \quad g \in G.
\end{equation*}
Hence, from the naturality of $\sigma: A \otimes - \to - \otimes A$, we get that
$\partial(Z) = \partial(X) \partial(Y)$ for every composition factor $Z$ of $X \otimes Y$. This proves the proposition. \end{proof}

Let $\C_0$ be the trivial homogeneous component of the $\Aut(G)$-grading of $\C$.
Thus, for every simple object $X$ of $\C_0$, $\sigma_X: A \otimes X \to X \otimes A$ is given componentwise by isomorphisms
\begin{equation}\label{sigma-g} \sigma^g_X: X_g \otimes X \to X \otimes X_g, \quad g \in G.
\end{equation}

\begin{remark} For a simple object $X_h$, $h \in G$, $\partial(X_h) \in \Aut(G)$ coincides with the inner automorphism induced by $h$ (c.f. \eqref{sigma-a}). In particular, $X_h$ belongs to $\C_0$ if and only if $h \in Z(G)$. Hence the functor $f$ induces an equivalence of tensor categories $f_0: \vect_{Z(G)} \to \langle A \rangle \cap \C_0$.
\end{remark}

Let $\D_0 \subseteq \D$ be the full subcategory whose objects are subquotients of $F(X)$, $X \in \C_0$.
Then $\D_0$ is a tensor subcategory of $\D$ and the tensor functor $F$ induces by restriction a dominant tensor functor $F_0: \C_0 \to \D_0$.
Moreover, $\KER_{F_0} = \langle A \rangle \cap \C_0$ and there is an exact sequence of finite tensor categories
\begin{equation}\label{zg}\vect_{Z(G)} \overset{f_0}\toto \C_0 \overset{F_0}\toto \D_0.
\end{equation}

Observe that since the center $Z(G)$ is an abelian group, then there is a strict   equivalence of tensor categories $\vect_{Z(G)} \cong \Rep \widehat{Z(G)}$. Hence \eqref{zg} is an abelian exact sequence of finite tensor categories (c.f. Example \ref{g-abelian}). As a consequence of Theorem \ref{main-ca}, $\C_0 \cong \D_0^{(Z(G), \Gamma)}$ for a suitable matched pair of groups $(Z(G), \Gamma)$.

\medbreak
The results of this section are summarized in the following:

\begin{theorem}\label{main-pt} Let \eqref{ext-pt} be an exact sequence of fini\textit{}te tensor categories. Then there exist a finite group $\Gamma$ endowed with mutual actions by permutations $\lhd: \Gamma \times Z(G) \to \Gamma$ and $\rhd: \Gamma \times Z(G) \to Z(G)$ making $(Z(G), \Gamma)$ a matched pair, and an $\Aut(G)$-grading on $\C$ whose neutral homogeneous component is a $(Z(G), \Gamma)$-crossed extension of a tensor subcategory of $\D$. \qed
\end{theorem}

Combining Theorem \ref{main-pt} with Corollary \ref{wgt-ext-rep} we obtain:

\begin{corollary}\label{wgt-pt} Let $\vect_G \toto \C \toto \D$ be an exact sequence of finite tensor categories. Then $\C$ is a weakly group-theoretical fusion category if and only if $\D$ is a weakly group-theoretical fusion category. \qed
\end{corollary}

\section{Semisolvable Hopf algebras}\label{section-hopf}

A short exact sequence of  Hopf algebras is a sequence of Hopf algebra maps
\begin{equation}\label{sec-hopf}k \toto H' \overset{i}\toto H \overset{\pi}\toto H'' \toto k,
\end{equation} such that

\medbreak (a)  $i$ is injective and  $\pi$ is surjective,

(b)  $\ker \pi = Hi(H')^+$, where $i(H')^+$ is the augmentation ideal of $i(H')$, and

(c)  $i(H') = {}^{\co \pi}H = \{h \in H: \, (\pi \otimes \id)\Delta (h) = 1 \otimes h\}$.

\medbreak 
Recall that every short exact sequence of finite dimensional Hopf algebras \eqref{sec-hopf}
induces an exact sequence of tensor categories
\begin{equation}\label{sec-rep}H''\!\!-\!\textrm{mod}  \overset{\pi^*} \toto H\!\!-\!\textrm{mod} \overset{i^*}\toto H'\!\!-\!\textrm{mod}.
\end{equation}
See \cite[Subsection 3.2]{tensor-exact}.

\medbreak 
A finite dimensional Hopf algebra $H$ is called \emph{weakly group-theoretical} if the category $H-$mod (or, equivalently, the category $\comod H$) is weakly group-theoretical.

\begin{remark}\label{quotient} Suppose that $H$ is a finite dimensional Hopf algebra that fits into an exact sequence \eqref{sec-hopf}. Note that the tensor functor $$\Hom_{H'}(k, i^* \pi^*(-)): H''\!\!-\!\textrm{mod} \to \vect$$ induced by \eqref{sec-rep} is monoidally isomorphic to the forgetful functor. 

Assume in addition that $H''$ is cocommutative. Then \eqref{sec-rep} is an abelian exact sequence of finite tensor categories. 
It follows from Corollary \ref{wgt-ext-rep}  that $H$ is  weakly group-theoretical Hopf algebra if and only if $H'$ is weakly group-theoretical.

Similarly, if $H''$ is commutative, then Corollary \ref{wgt-pt} implies that $H$ is weakly group-theoretical if and only if $H'$ is weakly group-theoretical.
\end{remark}

A \emph{lower subnormal
series} of a Hopf algebra $H$ is
a series of Hopf subalgebras
\begin{equation}\label{lowerseries}k = H_{n}  \subseteq H_{n-1}
\subseteq  \dots
\subseteq H_1 \subseteq H_0 = H, \end{equation} such that $H_{i+1}$ is a
normal Hopf subalgebra of $H_i$, for all $i$. The \emph{factors} of the series
\eqref{lowerseries} are the quotient Hopf algebras $\overline H_i =
H_i/H_iH_{i+1}^+$, $i = 0, \dots, n-1$.

Dually, an \emph{upper
subnormal series} of $H$ is a series of surjective Hopf
algebra maps
\begin{equation}\label{upperseries}H = H_{(0)} \to H_{(1)}
\to \dots \to H_{(n)} = k,
\end{equation} such that $H_{(i+1)}$ is a normal quotient Hopf algebra of
$H_{(i)}$, for all $i = 0, \dots, n-1$. The \emph{factors} of
\eqref{upperseries} are the Hopf algebras $\underline H_i = {}^{\co
H_{(i+1)}}H_{(i)} \subseteq H_{(i)}$, $i = 0, \dots, n-1$.

\medbreak 
The Hopf algebra $H$ is called \emph{lower-semisolvable} (respectively, \emph{upper-semisolva\-ble}) if $H$ admits a lower (respectively, upper) subnormal series whose factors are commutative or cocommutative \cite{MW}. We shall say that $H$ is \emph{semisolvable} if it is either lower or upper semisolvable. Every semisolvable finite dimensional Hopf algebra is semisimple and cosemisimple.

\medbreak
As a consequence of Corollaries \ref{wgt-ext-rep} and \ref{wgt-pt} we obtain:

\begin{corollary}\label{main-hopf} Let $H$ be a semisolvable finite dimensional Hopf algebra. Then $H$ is weakly group-theoretical. 
\end{corollary}

\begin{proof} Suppose first that $H$ is lower semisolvable and let \eqref{lowerseries} be a lower subnormal series of $H$ with commutative or cocommutative composition factors $\overline{H_i}$. Then $H$ fits into an exact sequence $k \toto H_{1} \toto H \toto \overline{H_0} \toto k$, where $H''$ is either commutative or cocommutative. Moreover, $H_{1}$ is also a lower semisolvable Hopf algebra. As pointed out in Remark \ref{quotient}, $H$ is weakly group-theoretical if and only if $H_{1}$ is weakly group-theoretical. The the corollary follows in this case by induction on $\dim H$.

Suppose next that $H$ is upper semisolvable. By \cite[Corollary 3.3]{MW}, $H^*$ is lower semisolvable and therefore it is weakly group-theoretical, by the previous part. Since $H^*\!\!-\!\textrm{mod} \cong \comod H$, then $H$ is weakly group-theoretical as well.
\end{proof}

\bibliographystyle{amsalpha}

\begin{thebibliography}{AAAA}
\bibitem{tensor-exact} A. Brugui\`{e}res, S. Natale, \emph{Exact sequences of
tensor categories},
Int. Math. Res. Not. \textbf{2011}  (24), 5644--5705 (2011).

\bibitem{indp-exact} A. Brugui\`{e}res, S. Natale, \emph{Central exact sequences
of tensor categories, equivariantization and applications}, J. Math. Soc. Japan \textbf{66}, 257--287 (2014).

\bibitem{DGNO} V. Drinfeld, S. Gelaki, D. Nikshych and
V. Ostrik, \emph{On braided fusion categories I}, Sel. Math. New
Ser. \textbf{16}, 1--119 (2010).

\bibitem{eg-emc}  P. Etingof, S. Gelaki, \emph{Exact sequences of tensor categories with respect to a module category},
Adv. Math. \textbf{308}, 1187--1208 (2017).

\bibitem{EGNO}  P. Etingof, S. Gelaki, D. Nikshych, V. Ostrik, \emph{Tensor categories},  Mathematical Surveys and Monographs \textbf{205},  Amer. Math. Soc., Providence, RI, 2015.

\bibitem{ENO2}  P. Etingof, D. Nikshych, V. Ostrik,
\emph{Weakly group-theoretical and solvable fusion categories},
Adv. Math \textbf{226},    176--205   (2011).

\bibitem{gel-nik} S. Gelaki, D. Nikshych, \emph{Nilpotent fusion categories}, Adv. Math. \textbf{217}, 1053--1071 (2008).

\bibitem{joyal-street} A. Joyal, R. Street, \emph{Braided tensor categories}, Adv. Math. \textbf{102} (1993) 20--78.

\bibitem{kac} G. I. Kac, \emph{Extensions of groups to ring groups}, Math. USSR
Sbornik \textbf{5} (1968), 451--474.

\bibitem{mk-ext} A. Masuoka,  \emph{Extensions of Hopf algebras}, Trab. Mat. \textbf{41/99}, FaMAF., Univ. Nac. de C\'  ordoba, 1999.

\bibitem{mn} M. Mombelli, S. Natale, \emph{Module categories over equivariantized tensor categories}, Mosc. Math. J. \textbf{17} (1), 97--128 (2017).

\bibitem{MW} S. Montgomery,  S. J. Witherspoon,
\emph{Irreducible representations of crossed products}, J. Pure
Appl. Algebra \textbf{129}, 315--326 (1998).

\bibitem{mueger-crossed} M. M\" uger,
\emph{Galois extensions of braided tensor categories
and braided crossed G-categories},
J. Algebra \textbf{277}, 256--281  (2004).

\bibitem{crossed-action} S. Natale, \emph{Crossed actions of matched pairs of groups on tensor categories}, Tohoku Math. J. \textbf{68} (3), 377--405 (2016).

\bibitem{core-wgt} S. Natale, \emph{The core of a weakly group-theoretical fusion category}, Int. J. Math. \textbf{29}, No. 2, Article ID 1850012, 23 p. (2018).

\bibitem{ostrik} V. Ostrik, \emph{Module categories, weak Hopf algebras and modular invariants}, Transform. Groups \textbf{8},  177--206 (2003).

\bibitem{pareigis} B. Pareigis, \emph{On braiding and dyslexia}, J. Algebra
\textbf{171}, 413--425 (1995).

\bibitem{schauenburg} P. Schauenburg, \emph{The monoidal center construction and
bimodules}, J. Pure Appl. Algebra \textbf{158}, 325--346 (2001).

\end{thebibliography}

\end{document}